\documentclass[a4paper,12pt]{article}

\usepackage[margin=3cm,footskip=1cm]{geometry}

\usepackage[T1]{fontenc}

\usepackage{amsmath,amssymb,amsthm,bm,mathtools,xspace,booktabs,url}
\usepackage{graphicx,etoolbox}

\usepackage{amsfonts,dsfont}
\usepackage[english] {babel}
\usepackage{array}

\usepackage[shortlabels]{enumitem}
\setlist{
  listparindent=\parindent,
  parsep=0pt,
}

\AtBeginEnvironment{thebibliography}{%
  \small
  \setlength\itemsep{0.75em plus 0.5em minus 0.75em}%
}
\usepackage{caption}
\usepackage[raggedright]{titlesec}

\numberwithin{equation}{section} 
\theoremstyle{plain} 
\newtheorem{theorem}{Theorem}[section]
\newtheorem{lemma}[theorem]{Lemma}

\newtheorem{definition}[theorem]{Definition}
\newtheorem{corollary}[theorem]{Corollary}

\theoremstyle{definition} 
\newtheorem{example}[theorem]{Example}
\newtheorem{remark}[theorem]{Remark}
\newtheorem{property}[theorem]{Property}
\usepackage{authblk}

\makeatletter
\newcommand\CorrespondingAuthor[1]{%
  \begingroup%
  \def\@makefnmark{}%
  \footnotetext{Corresponding author: #1}%
  \endgroup%
}
\makeatother

\makeatletter
\renewenvironment{abstract}{%
  \small%
  \begin{center}%
    \bfseries \abstractname\vspace{-.5em}\vspace{\z@}%
  \end{center}%
  \quote%
}{\endquote}
\makeatother

\usepackage[above,below,section]{placeins}

\usepackage[load-configurations=abbreviations]{siunitx}

\usepackage[all]{onlyamsmath}

\usepackage{lipsum}

\newcommand{\R}{\mathbb{R}}
\DeclareMathOperator{\supp}{supp}

\newcommand{\La}{\mathbb{L}}

\newcommand{\eps}{\varepsilon}
\DeclareMathOperator{\conv}{conv}

\DeclareMathOperator{\indre}{int}

\DeclareMathOperator{\tr}{Tr}
\newcommand{\Ha}{\mathcal{H}}

\newcommand{\II}{\textit{II}}

\DeclareMathOperator{\Err}{Err}

\begin{document}

\title{Estimation of intrinsic volumes from digital grey-scale images}

\author{Anne Marie Svane}

\affil[1]{Department of Mathematical Sciences\\ 
Aarhus University\\ amsvane@imf.au.dk}
\date{}

\maketitle

\begin{abstract}
Local algorithms are common tools for estimating intrinsic volumes from black-and-white digital images. However, these algorithms are typically biased in the design based setting, even when the resolution tends to infinity. Moreover, images recorded in practice are most often blurred grey-scale images rather than black-and-white. In this paper, an extended definition of local algorithms, applying directly to grey-scale images without thresholding, is suggested. We investigate the asymptotics of these new algorithms when the resolution tends to infinity and apply this to construct estimators for surface area and integrated mean curvature that are asymptotically unbiased in certain natural settings.
\end{abstract} 

\section{Introduction}
In this paper, we shall investigate the class of so-called local algorithms \cite{digital, am3} used for estimation of surface area and integrated mean curvature from  digital images. These algorithms are commonly used in applied sciences for analysing di\-gi\-tal output data from e.g.\ microscopes and scanners, see \cite{digital,lindblad,mecke2}. 
The main reason for the po\-pu\-la\-ri\-ty of local algorithms is that they allow simple linear time implementations~\cite{OM}. However, this efficiency is usually paid for by a lack of accuracy \cite{kampf2,am3}.

Local algorithms have so far only been defined for black-and-white images, see e.g.\ \cite{am2}. In a black-and-white image of a geometric object $X\subseteq \R^d$, each pixel is coloured black if the midpoint lies in $X$ and white otherwise. The use of local algorithms thus requires that we are able to measure precisely whether or not a given point belongs to $X$.  
In practice, however, such exact measurements are typically not possible, since the light coming from each point is spread out following a point spread function (PSF). The result is a grey-scale image where each pixel is assigned a grey-tone corresponding to a measured light intensity between 0 and 1. The standard way of overcoming this problem is to convert the image to black-and-white by choosing a threshold limit $\beta \in [0,1)$ and converting all pixels with grey-value greater than $\beta$ to black and all others to white. 

As a natural way of assessing a local algorithm, we test it in the design based setting where the object under study has been randomly translated with respect to the observer before the image is recorded. Ideally, the estimator should be unbiased, at least asymptotically when the resolution tends to infinity. For black-and-white images, surface area and integrated mean curvature can generally not be estimated without an asymptotic bias unless it coincides with the Euler cha\-rac\-te\-ris\-tic \cite{kampf2,am3,johanna}.
To make the asymptotic behaviour of grey-scale images well-defined it is necessary to make some assumptions on how the PSF changes with increased resolution. In this paper, we will assume that measurements become more accurate with increased resolution, see the precise assumption and a discussion of this in Subsection~\ref{models} below. 

Most previous asymptotic studies for black-and-white images do not take the thresholding process into account. For volume estimators, some first studies for simple PSF's were performed in \cite{hall,rataj} and in \cite{stelldinger}, estimators for the Euler characteristic are studied in 2D. The first purpose of this paper is to study the effect of thresholding on local algorithms.  

On the other hand, most of the information hidden in the grey-values is thrown away in the thresholding process. The second purpose of this paper is to give an extended definition of local algorithms, see Subsection \ref{lag-s}, that exploits the available information in the grey-scale image better but still leads to fast computations. Finally we are going to study the asymptotic bias of these estimators. 

%

\subsection{Results and applications}
The asymptotic studies of local algorithms will be based on three theoretical formulas extending \cite[Theorem 2]{rataj}. In Section \ref{1storder} this theorem is extended to larger classes of PSF's and in Section \ref{sof} it is extended to a second order formula. The techniques involved in the proofs are similar to those used in \cite{rataj} and \cite{am2}. 

From the theoretical formulas we obtain some applications to surface area estimation in Section \ref{app1}, and in  Section \ref{app2} we apply the results to estimators for the integrated mean curvature and the first order bias of surface area estimators in finite high resolution. We summarize some of the main findings here.

Assuming only mild conditions on the PSF, we first consider surface area estimators applied to the class of so-called gentle sets, see Definition~\ref{gentle} below. This class includes for instance all manifolds and all polyconvex sets. For black-and-white local algorithms applied to thresholded images, we find that the asymptotic bias is the same as for black-and-white images, see Subsection~\ref{thri}. In particular, local algorithms applied to thresholded images are not asymptotically unbiased.
 
In contrast to this, Subsection \ref{direct} shows that if the PSF is rotation invariant, asymptotically unbiased estimators based directly on the grey-values are plenty. As a simple example, the estimator counting the number of grey-values in some interval $I\subseteq (0,1)$ has the correct asymptotic mean up to a constant factor. Moreover, if $I$ is symmetric around $\frac{1}{2}$, the first order bias in high resolution vanishes. This algorithm is clearly both simple and fast and it can even be applied if the grey-values in the output data are only given discretely.

For more general classes of PSF's we obtain a description of the worst case asymptotic error. This could be used to search for estimators minimizing the asymptotic bias. 

With stronger assumptions on both the PSF and the smoothness of the boun\-da\-ry of the underlying set $X$, we find in Subsection \ref{imc} that if the PSF is rotation invariant, also asymptotically unbiased estimators for the integrated mean curvature do exist. One example is given by the estimator that counts the number of grey-values in the interval $(\beta,\frac{1}{2})$ and subtracts the number of grey-values in $(\frac{1}{2},1-\beta)$ for a suitable $\beta \in (0,\frac{1}{2})$. For thresholded images, the asymptotic mean is a little more complicated than in the black-and-white case, now depending on the PSF, see Subsection \ref{thri2}.

All results of this paper are theoretical. The practical usefulness of local algorithms for grey-scale images is discussed in the final Section \ref{discuss}.

\section{Local digital algorithms}
In this section we introduce local digital estimators for the surface area $2V_{d-1}(X)$ and integrated mean curvature $2\pi(d-1)^{-1}V_{d-2}(X)$ of a `sufficiently nice' set $X\subseteq \R^d$. Both quantities are examples of the so-called intrinsic volumes $V_q$, $q=0,\dots, d$, hence we shall use this term when referring to both, see e.g.\ \cite{schneider}.

\subsection{Models for digital images}\label{models}
Let $\La$ be the lattice in $\R^d$ spanned by the ordered basis $v_1,\dots,v_d\in \R^d$ and let  $C_v=\bigoplus_{i=1}^d[0,v_i]$ be the fundamental cell of the lattice. As we shall later be scaling the lattice, we may as well assume that the volume $\det(v_1,\dots,v_d)$ of $C_v$ is 1. For $c\in \R^d$, we let $\La_c=\La+c$ denote the translated lattice. 

We shall think of the pixels in a digital image as the translations of $C_v$ that have midpoints in $\La_c$.
Let $X\subseteq \R^d$ be a geometric object. The information about $X$ hidden in a black-and-white image corresponds to  the set of black pixel midpoints $X\cap \La_c$ . 

For grey-scale images, we assume that the light coming from each point is spread out following a point spread function  which is independent of the position of the point. That is, the light that reaches the observer is given by the intensity function
\begin{equation*}
\theta^{X,\rho} : \R^d \to [0,1]
\end{equation*}
where the intensity measured at $x\in \R^d$ is given by
\begin{equation*}
\theta^{X,\rho}(x)= \int_{X} \rho(z-x) dz.
\end{equation*}
Here $\rho$ is the PSF, which is assumed to be a measurable function satisfying $\rho \geq 0$ and $\int_{\R^d} \rho d\Ha^d =1 $ where $\Ha^d$ is the $d$-dimensional Hausdorff measure.
A digital image in the grey-scale setting is the restriction of $\theta^{X,\rho}$ to the observation lattice $\La_c$.

A simple example of a PSF is $\rho_B = \Ha^d(B)^{-1}\mathds{1}_{B}$ where $B\subseteq \R^d$ is a Borel set of non-zero finite volume and $\mathds{1}_B$ denotes the indicator function. At every point $x\in \R^d$, $\theta^{X,\rho_B}(x)$ measures the volume of $(x+B)\cap X$. For instance, if $B$ equals $C_0=C_v-\frac{1}{2}\sum_{i=1}^d v_i$, the grey-value of a pixel measures the fraction of the pixel that is contained in $X$. This PSF is studied thoroughly in 2D in \cite{hall}. Another interesting example is when $B$ is the ball $B(R)$ of radius $R>0$. It may also be relevant to consider various continuous approximations to these caused by imprecisions of measurements near the boundary of $B$. 

In other applications, it is more relevant to consider a PSF with non-compact support. The main example to have in mind is the Gaussian
\begin{equation*}
\rho_{Gauss}(x)=\frac{1}{(\sqrt{2\pi }\sigma)^d}\text{exp}\Big(-\frac{|x|^2}{2\sigma^2}\Big)
\end{equation*}
which yields a good approximation of most PSF's occurring in practice \cite{kothe}. 

We say that a PSF is reflection invariant if $\rho(x)=\rho(-x)$ and rotation invariant if $\rho(x)$ depends only on $|x|$. For instance,  $\rho_{B(R)}$ and $\rho_{Gauss}$ are rotation invariant, while $\rho_{C_0}$ is only reflection invariant.

A change of resolution to $a^{-1}$ for some $a>0$ corresponds to a change of lattice from $\La$ to $a\La$. We assume that the precision of the measurements changes in such a way that the PSF in resolution $a^{-1}$ is
\begin{equation*}
\rho_a(x)=a^{-d}\rho(a^{-1}x).
\end{equation*}
The corresponding intensity function is denoted
\begin{equation*}
\theta_a^{X,\rho}(x)= \int_{X} \rho_a(z-x) dz =a^{-d} \int_{X} \rho(a^{-1}(z-x)) dz.
\end{equation*}
The PSF is omitted from the notation whenever it is clear from the context. 

In some applications, e.g.\ for $\rho_B$ or in some cases where the blurring is caused by the optical device, this transformation of $\rho$ with the resolution is natural. For $\rho_{C_0}$ it simply means that pixels become smaller in higher resolution. In other situations, e.g.\ if the light is spread out before it reaches the lense, it may be impossible for the observer to affect the blurring or a different transformation is more realistic. However, in this paper we restrict ourselves to the above setting.

\subsection{Local algorithms for black-and-white images}
We first recall the definition of local algorithms in the case of black-and-white images, see e.g.\ \cite[Section 2]{am3} for more details and justifications of such algorithms. 

An $n\times \dots \times n$ lattice cell is a set of the form $C_z^n=(z+\bigoplus_{i=1}^d [0,nv_i))$ where $z\in \La$. The set of lattice points lying in such a cell is denoted by $C_{z,0}^n=C_z^n\cap \La $. 
%
%
An $n\times \dotsm \times n$ configuration is a pair $(B,W)$ where $B, W \subseteq C_{0,0}^n$ are disjoint with $B\cup W = C_{0,0}^n$. 
We index these by $(B_l,W_l)$ for $l=0,\dots,2^{n^d}-1$ where $B_0=W_{2^{n^d}-1}=\emptyset$. 

A local digital algorithm in the sense of \cite{am3} estimates $V_q$ by a weighted sum of configuration counts:
\begin{definition}\label{funddef}
A local digital estimator $\hat{V}_q$ for $V_q$ based on the image $X\cap a\La_c$ is an estimator of the form
\begin{equation*}\label{funddefeq}
\hat{V}_q^{ a\La_c}(X)= a^q\sum_{l=1}^{2^{n^d}-1} w_l N_l^{a\La_c}(X )
\end{equation*} 
where 
\begin{equation*}
N_l^{a\La_c}(X )= \sum_{z \in a\La_c }  \mathds{1}_{\{z+aB_l\subseteq X\}}\mathds{1}_{\{ z+ aW_l \subseteq \R^d \backslash X\}}
\end{equation*} 
is the total number of occurrences of the configuration $(B_l,W_l)$ in the image and the constants $w_l \in \R$ are called the weights.
\end{definition}
Note that we assume the weights to be homogeneous in the sense of \cite{am3}.
In  all the algorithms used in practice, the weights are homogeneous. 

Suppose a grey-scale digital image is thresholded at level $\beta \in [0,1)$. This means that the set of black lattice points is now $\{z\in a\La_c \mid \theta^X_a(z)>\beta \}$. Replacing $X\cap a\La_c$  by this set in Definition \ref{funddefeq},
the resulting estimator becomes
\begin{equation}\label{thresholdV}
\hat{V}(\beta)_q^{ a\La_c}(X)= a^q\sum_{l=1}^{2^{n^d}-1} w_lN(\beta)_l^{a\La_c}(X) 
\end{equation} 
where
\begin{equation*}
N(\beta)_l^{a\La_c}(X) = \sum_{z \in a\La_c }  \mathds{1}_{ \{\theta^X_a(z + a{B}_l)\subseteq (\beta,1]\}}\mathds{1}_{\{ \theta^X_a(z + a{W}_l)\subseteq [0,\beta]\}}.
\end{equation*} 

\subsection{Local algorithms in the grey-scale setting}\label{lag-s}
We now suggest a more general definition of local algorithms based directly on grey-scale images. 
An $n\times \dots \times n$ configuration in the grey-scale setting is a point $\theta_a^X(aC_{z,0}^n) \in [0,1]^{n^d}$. Each configuration must thus be weighted not by a finite collection of weights but by a function $f:[0,1]^{n^d}\to \R$. This leads to the following definition:
\begin{definition}\label{greyest}
A local estimator $\hat{V}_q$ for $V_q$ is an estimator of the form 
\begin{equation*}
\hat{V}(f)_q^{a\La_c}(X)=a^q\sum_{z \in \La_c} f(\theta^X_a(aC_{z,0}^n))
\end{equation*}
where $f:[0,1]^{n^d} \to \R$ is a Borel function.


To ensure finiteness of estimators on compact sets, we assume that $f(0)=0$ and, if $\rho $ has non-compact support, we also assume $\supp f \subseteq (0,1]$. 

To ensure integrability of $z \mapsto f\circ \theta^X_a(aC_{z,0}^n)$ when $X$ is compact, we moreover assume that $f$ is bounded.
\end{definition}

In the definition, one could of course let $f$ depend on both lattice distance $a>0$ and position $z\in \R^d$, but due to the homogeneity and translation invariance of intrinsic volumes, we restrict ourselves to the estimators in  Definition \ref{greyest}. 

Algorithms of this type have already been considered in \cite{hall} and \cite{mantz}. Also \eqref{thresholdV} is a special case of this, but Definition \ref{greyest} allows a much more refined use of the grey-values.
Apart from \eqref{thresholdV}, we shall mainly consider estimators with $n=1$, corresponding to estimators of the form
\begin{equation}\label{n1}
\hat{V}(f)_q^{a\La_c}(X)=a^q\sum_{z\in \La_c} f(\theta^X_a(az))
\end{equation}
where $f:[0,1] \to \R$ is as in Definition \ref{greyest}.

\subsection{Convergence in the design based setting}
We shall investigate local algorithms for digital grey-scale images in the design based setting. This may be modeled as the situation where $X\subseteq \R^d$ is held fixed while the observation lattice $\La_c$ is the translation of $\La $ by a uniform random translation vector $c\in C_v$. The mean of a local estimator applied to a grey-scale  image is then
\begin{equation}\label{mean}
E\hat{V}(f)_q^{a\La_c}(X)=a^qE\sum_{z\in \La_c} f(\theta^X_a(aC_{z,0}^n)) = a^{q-d}\int_{\R^d}f \circ \theta^X_a(z+aC_{0,0}^n) \Ha^d(dz).
\end{equation}

As a natural way of assessing a local algorithm in the design based setting, we study the mean estimator when the resolution goes to infinity. Ideally, the algorithm would be asymptotically unbiased, i.e.
\begin{equation*}
\lim_{a\to 0} E\hat{V}(f)_q^{a\La_c}(X)={V}_q(X)
\end{equation*}
for all sets $X$ in some family $\mathcal{S}$ of subsets of $\R^d$.

In the case of surface area estimation, we more generally consider the worst case asymptotic relative mean error as a measure for how well an algorithm works:
\begin{equation*}
\text{Err}(\hat{V}(f)_{d-1})=\sup_{X\in \mathcal{S}}\bigg|\frac{\lim_{a\to 0}E\hat{V}(f)_{d-1}^{a\La_c}(X)-{V}_{d-1}(X)}{{V}_{d-1}(X)}\bigg|.
\end{equation*}

\section{First order formulas}\label{1storder}
We first derive some abstract formulas, from which the first order asymptotic behaviour of the integral in \eqref{mean} can be determined.
These extend the formula by Kiderlen and Rataj given  in \cite[Theorem 2]{rataj} for PSF's of the form $\rho_B$ to some larger classes of PSF's. Though only considered in their paper as a correction term to volume estimators, the formulas have applications to surface area estimation as well. 
We first introduce some notation and state their results in the language of the present paper. 

For a closed set $X\subseteq \R^d$, we let $\text{exo}(X)$ denote the points in $\R^d$ not having a unique nearest point in $X$. Let $\xi_X : \R^d\backslash \text{exo}(X) \to X$ be the natural projection taking a point in $\R^d\backslash \text{exo}(X)$ to its unique nearest point in $X$. We define the normal bundle of $X$ to be the set
\begin{equation*}
N(X)=\big\{\big(x,\tfrac{z-x}{|z-x|}\big)\in X\times S^{d-1}\, \big|\, z\in \R^d\backslash (X\cup \text{exo}(X)),\, \xi_X(z)=x \big\}.
\end{equation*}
For $(x,n)\in N(X)$ we define the reach
\begin{equation*}
\delta(X;x,n)=\inf\{t\geq 0 \mid x+tn \in \text{exo}(X)\}\in (0,\infty].
\end{equation*}

Following \cite{rataj}, we introduce the class of gentle sets:
\begin{definition}\label{gentle}
A closed set $X\subseteq \R^d $ is called gentle if
\begin{itemize}
\item[(i)] $\Ha^{d-1}(N(\partial X) \cap (B \times S^{d-1}))<\infty$ for any bounded Borel set $B\subseteq \R^d$.
\item[(ii)] For $\Ha^{d-1}$--almost all $x\in \partial X$ there exist two balls $B_{in},B_{out}\subseteq \R^d$ both containing $x$ and such that $B_{in}\subseteq X$, $\indre(B_{out})\subseteq \R^d \backslash X$.
\end{itemize}
\end{definition}
The condition (ii) in the definition means that for almost all $x\in \partial X$ there is a unique pair $(x,n(x))\in N(X)$ with $(x,n(x)),(x,-n(x))\in N(\partial X)$.

For $n\in  S^{d-1}$, we let $H_{t,n}^-$ denote the halfspace  $\{x\in \R^d\mid \langle x,n\rangle \leq t\}$. For short we sometimes write $H_n=H_{0,n}^-$. For a set $S\subseteq \R^d$, $h(S,n)=\sup \{\langle s,n\rangle\mid {s\in S}\}$ denotes the support function and $\check{S}=\{-s\mid s\in S\}$. Finally, $\oplus$ denotes Minkowski addition and for $t\in \R$, we use the notation $t^+=t\vee 0 =\max\{t,0\}$. 

We are now ready to state the first order formula \cite[Theorem 2]{rataj}:
\begin{theorem}[Kiderlen, Rataj]\label{knownform}
Let  $X \subseteq \R^d$ be a closed gentle set and $A\subseteq \R^d$ be bounded measurable. Let $ \beta,\omega \in (0, 1]$ and  $B,W,P,Q\subseteq \R^d$ non-empty compact with $\Ha^{d}(P),\Ha^d(Q)>0$. Then
\begin{align*}
\MoveEqLeft \lim_{a\to 0} a^{-1}\int_{\xi^{-1}_{\partial X}(A)}  \mathds{1}_{\big\{\theta_a^{X,\rho_P}(x+aB)\subseteq[\beta,1]\big\}} \mathds{1}_{\big\{ \theta_a^{X,\rho_Q}(x+aW)\subseteq [0,\omega)\big\}}dx \\
&=\int_{\partial X \cap A} (\tilde{\varphi}^{\rho_{P}}(\beta,n)-\tilde{\varphi}^{\rho_{Q}}(\omega,n)-h({B}\oplus \check{W},n))^+d\Ha^{d-1}
\end{align*}
where
\begin{equation*}
\tilde{\varphi}^{\rho}(\beta,n) = \sup\{t\in\R \mid \theta^{H_n,\rho}(tn) \geq \beta \}  
\end{equation*}
for $n\in S^{d-1}$ and $\beta \in (0,1]$.
\end{theorem}

Observe in the definition of $\tilde{\varphi}^{\rho}$ that the continuous function 
\begin{equation*}
t\mapsto \theta^{H_n,\rho}(tn) = \theta^{H^-_{-t,n},\rho}(0)=\theta^{H_n,\rho}_a(atn)
\end{equation*}
is decreasing so that $\tilde{\varphi}^{\rho}(\beta,n) $ is finite decreasing for $\beta \in (0,1]$.


\subsection{The case of compact support}
We first generalize Theorem \ref{knownform} to PSF's that are almost everywhere bounded and compactly supported. Note how the open and closed ends of the intervals have been switched in the statement of the theorem. For this reason, the functions $\tilde{\varphi}$ are replaced by
\begin{equation*}
\varphi^{\rho}(\beta,n) = \text{inf}\{t\in \R \mid \theta^{H_n,\rho}(tn) \leq \beta \} = \text{sup}\{t \in \R \mid \theta^{H_n,\rho}(tn) > \beta \}.
\end{equation*}

\begin{theorem}\label{first1}
Let $X\subseteq \R^d$ be a closed gentle set and $A\subseteq \R^d$ bounded measurable. Let $I$ and $J$ be non-empty finite index sets. For $i\in I$ and $j\in J$, let $ \beta_i ,\omega_j \in [0, 1)$, $B_i,W_j\subseteq \R^d$ be non-empty compact, and $\rho_i,\rho_j$ be almost everywhere bounded PSF's with compact support. Then
\begin{align*}
\MoveEqLeft \lim_{a\to 0} a^{-1} \int_{\xi^{-1}_{\partial X}(A)} \prod_{i\in I}\mathds{1}_{\big\{\theta_a^{X,\rho_i}(x+aB_i) \subseteq (\beta_i, 1] \big\}} \prod_{j\in J} \mathds{1}_{\big\{\theta_a^{X,\rho_j}(x+aW_j) \subseteq [0,\omega_j]\big\} }dx \\
&=\int_{\partial X\cap A} (\min_{i\in I}\{\varphi^{\rho_i}(\beta_i,n)-h({B}_i,n)\} - \max_{j\in J}\{\varphi^{\rho_j}(\omega_j,n)+h(\check{W}_j,n)\})^+d\Ha^{d-1}.
\end{align*}
\end{theorem}

In the proof we shall use the following notation when a fixed $x\in \partial X$ is understood. 
We write $H:=H_{\langle x,n(x)\rangle,n(x)}^-$ and after possibly shrinking $B_{in}$ and $B_{out}$, we assume that they both have the common radius $r>0$.

For $t\in \R$ and $s\in \R^d$ we write
\begin{equation*}
\theta_a^X(t;s):=\theta^X_a(x + a(tn +s)).
\end{equation*}
The map $(t,s)\mapsto \theta_a^{X} (t;s)$ is continuous when $X$ is gentle since $\partial X \oplus B(\eps)\downarrow \partial X$ as $\eps \downarrow 0$ and $\Ha^{d}(\partial X)=0$ so that by monotone convergence,
\begin{align*}
|\theta_a^{X} (t';s') -\theta_a^{X} (t;s)| &=\bigg|\int_{ X-a(t'n+s')} \rho_a(z) dz-\int_{X-a(tn+s)} \rho_a(z) dz\bigg|\\
&\leq \bigg|\int_{\partial X \oplus B(a(|t-t'|+|s-s'|))} \rho_a(z-(x + a(tn +s))) dz\bigg|
\end{align*}
goes to 0 for $(t',s')\to (t,s)$.

For $x\in \partial X$ and $s\in S\subseteq \R^d$, let
\begin{align*}
t^X_+(a,\beta;s)&=\inf \{t\in (-\delta(\partial X, x,-n),\delta(\partial X,x,n)) \mid \theta_a^{X} (t;s) \leq \beta\}\\ 
t^X_-(a,\beta;s)&=\sup \{t \in (-\delta(\partial X, x,-n),\delta(\partial X,x,n)) \mid \theta_a^{X} (t;s) > \beta\}\\
t^X_+(a,\beta;S)&=\sup \{t^X_+(a,\beta;s) \mid {s\in S}\} \\
t^{X}_-(a,\beta;S)&=\inf \{t^X_-(a,\beta;s) \mid {s\in S}\}.
\end{align*}
When $t \mapsto \theta_a^{X} (t;s)$ is decreasing, we write
\begin{equation*}
t^X(a,\beta;s)=t^X_+(a,\beta;s)=t^X_-(a,\beta;s).
\end{equation*}

Since $\theta_a^H(t;s)$ is independent of $a$, we sometimes just write $\theta_0^H(t;s)$, $t^H(0,\beta;s)$, etc. Moreover, $\theta_0^H(t;s)=\theta_0^H(t+\langle s,n \rangle;0)$, so 
\begin{align*}
t^H_+(0,\beta;S)&=\varphi(\beta,n)+h(\check{S},n)\\
t^H_-(0,\beta;S)&=\varphi(\beta,n)-h(S,n).
\end{align*}

\begin{proof}[Proof of Theorem \ref{first1}]
For any gentle set $Y\subseteq \R^d$, let
\begin{align*}
f_a^Y(x)&= \prod_{j\in J} \mathds{1}_{\big\{\theta_a^{Y,\rho_j}(x+aW_j) \subseteq [0,\omega_j]\big\} }\\
g_a^Y(x)&= \prod_{i\in I}\mathds{1}_{ \big\{\theta_a^{Y,\rho_i}(x+aB_i) \subseteq (\beta_i, 1]\big\} }.
\end{align*}

Let $D$ be such that $B_i,W_j, \supp \rho_i ,\supp \rho_j  \subseteq B(D)$ for all $i,j$. This ensures that $\supp  f_a^Xg_a^X \subseteq \partial X \oplus aB(2D)$ and hence by  \cite[Theorem 2.1]{last} that
\begin{align}\label{lastform}
\MoveEqLeft \int_{\xi^{-1}_{\partial X}(A)}f_a^X(x)g_a^X(x)dx = \sum_{m=1}^d m\kappa_m \int_{N(\partial X)}\mathds{1}_A(x)\\&
\times \int_0^{\delta(\partial X;x,n)} t^{m-1} \nonumber
 f_a^X(x+tn)g_a^X(x+tn)dt \mu_{d-m}(\partial X;d(x,n))
\end{align}
where $\kappa_m$ is the volume of the unit ball in $\R^m$ and $\mu_{m}(\partial X,\cdot)$ are certain signed measures of locally finite total variation.

First observe that
\begin{align*}
\int_0^{\delta(\partial X;x,n)} t^{m-1} f_a^X(x+tn)g_a^X(x+tn)dt &\leq \int_0^{ 2aD} t^{m-1} f_a^X(x+tn)g_a^X(x+tn)dt\\
& \leq m^{-1}a^m(2D)^m. 
\end{align*}
Since each $\mu_{d-k}$ has locally finite total variation and $A$ is bounded, Lebesgue's theorem of dominated convergence and the identification of $\mu_{d-1}$ given in \cite[Equation~(8)]{rataj} yields
\begin{align*}
\MoveEqLeft \lim_{a\to 0}a^{-1}\sum_{m=1}^d m\kappa_m \int_{N(\partial X)}\mathds{1}_A(x) \\ &\times\int_0^{2aD} t^{m-1} f_a^X(x+tn)g_a^X(x+tn)dt \mu_{d-m}(\partial X;d(x,n)) \\
 ={}&  \lim_{a\to 0} a^{-1}\kappa_1 \int_{N(\partial X)} \mathds{1}_A(x) \int_0^{2aD} f_a^X(x+tn)g_a^X(x+tn)dt \mu_{d-1}(\partial X;d(x,n))\\
 ={}& \int_{\partial X\cap A}\bigg( \lim_{a\to 0}\int_{-2D}^{2D} f_a^X(x+atn)g_a^X(x+atn)dt\bigg) \Ha^{d-1}(dx)
\end{align*}
if the limit exists. Thus we consider the inner integral for $x\in \partial X $ fixed.

Observe that
\begin{equation*}
\theta_a^{B_{in}}(t;s) \leq \theta_a^X(t;s), \theta_a^{H}(t;s) \leq \theta_a^{\R^d \backslash B_{out}}(t;s).
\end{equation*}
Thus 
\begin{equation}
\begin{split}\label{ineq2}
f_a^{B_{in}}(x+atn)\geq f_a^X(x+atn)\geq f_a^{\R^d\backslash B_{out}}(x+atn)\\ 
g_a^{B_{in}}(x+atn)\leq g_a^X(x+atn)\leq g_a^{\R^d\backslash B_{out}}(x+atn)
\end{split}
\end{equation}
and hence
\begin{align}
\label{ineq}
\int_{-2D}^{2D}f_a^{\R^d\backslash B_{out}}(x+atn) g_a^{B_{in}}(x+atn)dt &\leq\int_{-2D}^{2D} f_a^X(x+atn) g_a^X(x+atn)dt\\ \nonumber
&\leq \int_{-2D}^{2D}f_a^{B_{in}}(x+atn)g_a^{\R^d\backslash B_{out}}(x+atn)dt.
\end{align}

If $Y_i$, $i=1,2$, are gentle sets for which $t\mapsto \theta_a^{Y_i}(t;s)$ are decreasing on $(-2D,2D)$ for all $s$ with $|s|\leq D$, then
\begin{align}\nonumber
\MoveEqLeft \int_{-2D }^{2D} f_a^{Y_1}(x+atn)g_a^{Y_2}(x+atn) dt\\\label{mint}
& = \int_{\max_{j\in J} \{t^{Y_1}_+(a,\omega_j;W_j) \}}^{\min_{i\in I}\{ t_-^{Y_2}(a,\beta_i;B_i)\}}  \mathds{1}_{\{\min_{i\in I} \{t^{Y_2}_-(a,\beta_i;B_i)\} >\max_{j\in J} \{t^{Y_1}_+(a,\omega_j;W_j)\}}dt\\
&= (\min_{i\in I}\{ t^{Y_2}_-(a,\beta_i;B_i) -\max_{j\in J} \{t^{Y_1}_+(a,\omega_j;W_j)\})^+.\nonumber
\end{align} 

From now on, we assume $2aD\leq r$. This guarantees that 
\begin{equation*}
x+a(B(2D)\oplus [-2Dn,2Dn]) \subseteq \conv(B_{in}\cup B_{out})
\end{equation*}
and thus \eqref{mint} holds for $Y_i$ equal to  $ {B_{in}}$, $H$, or $\R^d\backslash B_{out}$. Here $[x,y]$ denotes the line segment between $x$ and $y$.
Moreover,
\begin{align*}
(H \cap (x+a(tn +s + B(D)))-a\nu n) & \subseteq B_{in} \cap (x+a(tn +s+ B(D)))\\
\R^d \backslash B_{out} \cap (x+a(tn +s + B(D))) & \subseteq (H\cap (x+a(tn +s + B(D)))+a\nu n )
\end{align*}
whenever $a\nu \geq r-\sqrt{r^2-a^2(2D)^2}$.
It follows that
\begin{gather*}
\theta_a^H(t+\nu;s) \leq \theta_a^{B_{in} }(t;s ) \leq \theta_a^H(t;s )\\
\theta_a^H(t;s )\leq \theta_a^{\R^d\backslash B_{out}}(t;s ) \leq \theta_a^H(t-\nu;s)
\end{gather*}
for such $\nu$. Thus 
\begin{align}\nonumber
|t^H(a,\beta;s)-t^{B_{in}}(a,\beta;s)|,|t^H(a,\beta;s)-t^{\R^d \backslash B_{out}}(a,\beta;s)|&\leq a^{-1}(r-\sqrt{r^2-(2aD)^2})\\ \label{unift}
&\leq Ma
\end{align}
for some $M>0$ depending only on $r$ and $D$. Therefore,
\begin{equation*}
\bigg\lvert \int_{-2D }^{2D} f_a^{B_{in}}(x+atn)g_a^{\R^d\backslash B_{out}}(x+atn) dt-\int_{-2D }^{2D} f_a^{H}(x+atn)g_a^{H}(x+atn) dt \bigg\rvert \in O(a).
\end{equation*}

But for all $a$,
\begin{align*}
\MoveEqLeft \int_{-2D}^{2D} f_a^{H}(x+atn)g_a^{H}(x+atn) dt  = (\min_{i\in I} t_-^{H}(0,\beta_i;B_i) -\max_{j\in J} t^{H}_+(0,\omega_j;W_j))^+\\
&= (\min_{i\in I}\{\varphi^{\rho_i}(\beta_i,n)-h({B}_i,n)\} - \max_{j\in J}\{\varphi^{\rho_j}(\omega_j,n)+h(\check{W}_j,n)\})^+.
\end{align*}

Thus, the right hand side of \eqref{ineq} is forced to converge with this limit. A similar argument applied to the left hand side finally forces the middle term to converge with this limit as well, proving the theorem.
\end{proof}

\begin{remark}\label{tilderem}
If the one or more of the intervals $[0,\beta_i]$ or $(\omega_j,1]$ are replaced by $[0,\beta_i)$ or $[\omega_j,1]$, respectively, with $ \beta_i,\omega_j \in (0,1]$, the theorem clearly holds with the corresponding $\varphi$ replaced by $\tilde{\varphi}$ by a similar argument, as long as the intersection of all the intervals does not contain either of  0 or 1.
\end{remark}

\subsection{Generalization to PSF's with non-compact support}
The proof of Theorem \ref{first1} generalizes to the case where $\supp \rho$ is non-compact and satisfies:
\begin{property}\label{rhocond}
There is a $C>0$ such that $\rho\leq C$ almost everywhere and the function $t \mapsto -\int_{\partial H_{n}}\rho(z-tn)dz$ is continuous  for every $n\in S^{d-1}$.
\end{property}
The condition is satisfied for most PSF's occurring in practice, e.g.\ if $\rho $ is bounded and $\rho(z)\in O(|z|^k)$ for some $k<-d$.
Property \ref{rhocond} clearly ensures:

\begin{lemma}\label{lemH}
Let $\rho$ have Property \ref{rhocond}. Then $t \mapsto \theta^{H_{n}}(tn)$ is decreasing $C^1$ with 
\begin{equation*}
\frac{d}{dt}\theta^{H_n}(tn) = -\int_{\partial H_{n}}\rho(z-tn)dz.
\end{equation*}
\end{lemma}

We first need some technical lemmas. Let 
\begin{equation*}
\mu(R)= \int_{|z|\geq R} \rho(z) dz .
\end{equation*}
By integrability of $\rho$, $\lim_{R\to \infty}\mu(R)=0$.

\begin{lemma}\label{lemr}
Let $x\in \partial X$ and $a>0$ be fixed. Let  $K>0$ and $0 < R < r-2aK$, and suppose $\rho $ is a PSF with $\rho \leq C$ almost surely for some $C >0$. Then there is a constant $M>0$ depending only on $r$, $C$, and $K$ such that for all $t\in \R$ and $s\in \R^d$ with $|s|,|t|\leq K$,
\begin{gather*}
0\leq \theta^{H}_a(t;s)-\theta^{B_{in}}_a(t;s)\leq Ma^{-d}(R+a|s|)^{d+1}+ \mu(a^{-1}R)\\
0 \leq \theta^{\R^d\backslash B_{out}}_a (t;s) - \theta^{H}_a(t;s) \leq Ma^{-d}(R+a|s|)^{d+1}+ \mu(a^{-1}R).
\end{gather*}
\end{lemma}

\begin{proof}
Observe that $R$ is chosen so small that 
\begin{equation*}
x+a(tn+s)+B(R)\subseteq \conv(B_{in}\cup B_{out})
\end{equation*}
 whenever $|s|,|t|\leq K $. Since
\begin{equation*}
B_{in }\subseteq H \subseteq B_{in} \cup A \cup  \R^d \backslash (x+a(tn+s)+B(R))
\end{equation*}
where 
\begin{equation*}
A=(H \backslash B_{in})\cap (x+a(tn+s)+B(R)),
\end{equation*}
this ensures that
\begin{align*}
 \theta_a^{H}(t;s)
\leq {}& \int_{B_{in}} \rho_a(z-(x+a(tn +s)) dz 
+a^{-d} \int_{A} \rho(a^{-1}(z-(x+a(tn +s) ))) dz \\
&+ a^{-d}\int_{\R^d \backslash (x+a(tn+s)+B(R))} \rho(a^{-1}(z-(x+a(tn +s))))dz \\
\leq {}&\theta^{B_{in}}_a(t;s)) + a^{-d} C \int_{B^{d-1}(R+a|s|)} (r-\sqrt{r^2-|z|^2})dz+ \mu(a^{-1}R)\\
\leq {}&\theta^{B_{in}}_a(t;s) + Ma^{-d}(R+a|s|)^{d+1}+ \mu({a^{-1}R})
\end{align*}
where $B^{d-1}(D)$ denotes the ball in $\R^{d-1}$ of radius $D$.

The second inequality is similar, using
\begin{equation*}
\R^d\backslash B_{out} \subseteq H  \cup  ( \R^d \backslash (x+a(tn+s)+B(R)))\cup B 
\end{equation*}
with $B=(\R^d \backslash (B_{out}\cup H) )\cap  (x+a(tn+s)+B(R))$.
\end{proof}

When $n\in S^{d-1}$ and $s\in \R^d$ are given, we shall say that $\beta \in (0,1)$ is a regular value for $\theta^{H_n}_0(\cdot;s)$ if $t\mapsto \theta^{H_n}_0(t;s)=\theta^{H_n}_0(t+\langle s,n\rangle;0) $ has non-zero derivative on the set 
$\{t\in \R \mid  \theta^{H_n}_0(t;s)=\beta\}$.

\begin{lemma}\label{tsqueeze}
Let $x\in \partial X$ be fixed. Let $\rho$ be a PSF satisfying Property \ref{rhocond} and $B \subseteq \R^d$ a compact set. Suppose $\beta \in (0,1)$ is a regular value for $\theta^H_0(\cdot ;s)$ for all $s\in B$ and let $K>0$ be given. Then there is a function $\gamma(a)\in o(1)$ such that for all $s\in B$,
\begin{align*}
& \theta^{\R^d\backslash B_{out}}_a {}(t;s) < \beta & &\text{ for all } && t\in ( t^{H}(0,\beta;s)+\gamma(a),K ]\\
&\theta^{B_{in}}_a {}(t;s)>\beta & &\text{ for all } && t \in [-K, t^{H}(0,\beta;s)-\gamma(a))
\end{align*}
whenever $a$ is sufficiently small.
In particular, for $a$ small enough 
\begin{equation*}
|t^{B_{in}}_\pm(a,\beta;s)-t^{H}(0,\beta;s)|,|t^{\R^d\backslash B_{out}}_\pm(a,\beta;s)-t^{H}(0,\beta;s)| \leq \gamma(a)
\end{equation*}
whenever $s\in B$.
\end{lemma}

\begin{proof} 
Since $g(t,s) = \frac{d}{dt} \theta^{H}_0(t; s )$ is continuous and $g(t^{H}(0,\beta;s),s)<0 $ for all $s\in B$ by assumption, there is a $\delta>0$ such that
\begin{equation*}
M_1=\inf \{-g(t^{H}(0,\beta;s)+\xi,s) \mid s\in B, |\xi|\leq \delta \}>0 .
\end{equation*} 
Let $s \in B$, write $t_0=t^{H}(0,\beta;s)$, and let $|\nu|\leq \delta$. By the mean value theorem there is a $|\xi| \leq \nu$ such that
\begin{align*}
\theta^{H}_0 {}&(t_0;s) -\theta^{H}_0(t_0+\nu;s) = -\nu g(t_0+\xi,s) \geq M_1 \nu.
\end{align*}

Put $R(a)=a^{\eps}$ where $\frac{d}{d+1} < \eps < 1$.  Lemma \ref{lemr} shows that
\begin{equation*}
0\leq \theta_a^{\R^d\backslash B_{out}}(t;s)-\theta_0^{H}(t;s) \leq M_2 a^{-d}(R(a)+a|s|)^{d+1}+ \mu(a^{-1}R(a))
\end{equation*}
for all $t\in [-K,K]$.
Thus, whenever
\begin{equation*}
\delta\geq \nu > M_1^{-1}  (M_2 a^{-d}(R(a)+a|s|)^{d+1}+ \mu(a^{-1}R(a))), 
\end{equation*}
we get
\begin{equation*}
\beta >  M_2 a^{-d}(R(a)+a|s|)^{d+1}+ \mu(a^{-1}R(a)) + \theta^{H}_0(t_0+\nu;s) \geq \theta_a^{\R^d\backslash B_{out}}(t_0+\nu;s)
\end{equation*}
and since $\nu \mapsto \theta^{H}_0(t+\nu; s)$ is decreasing, this also holds when $\nu > \delta$ as long as $t_0+\nu \leq K$.

Thus we may take $\gamma(a)$ to be
\begin{equation*}
\gamma(a)=M_1^{-1} ( M_2 a^{-d}(R(a)+a\sup\{|s| , {s\in B}\})^{d+1}+ \mu(a^{-1}R(a))).
\end{equation*}
Then $\gamma(a) \in  o(1)$ since $\eps(d+1)-d>0$ and $\lim_{a\to 0} \mu(a^{\eps - 1})=0$.
The claim about $\theta_a^{ B_{in}}$ is similar.

For the last claim, choose $D>0$ such that $\mu(D) < \beta, 1-\beta $ and $B\subseteq B(D)$. Then for $a$ small enough,
\begin{equation*}
x+a(tn+s+B(D))\subseteq B_{in} \subseteq H
\end{equation*}
for all $t\in[-a^{-1}r,-2D]$, so for such $t$, 
\begin{equation*}
\theta_0^{H}(t;s)\geq \theta_a^{B_{in}}(t;s)\geq 1-\mu(D) > \beta.
\end{equation*}
The claim for $t^{B_{in}}_\pm(a,\beta;s)$ now follows from the first part with $K$ replaced by $2D$. The claim about $t^{\R^d\backslash B_{out}}_\pm(a,\beta;s)$ is similar.
\end{proof}

We can now state the main theorem for PSF's with non-compact support:
\begin{theorem}\label{first2}
Theorem \ref{first1} also holds for PSF's $\rho_i,\rho_j$ having non-compact support and satisfying Property \ref{rhocond} if  $\beta_i,\omega_j \in (0,1)$ are regular values for $\theta^{H_n,\rho_i}_0(\cdot ;b)$ and $\theta^{H_n,\rho_j}_0(\cdot ;w)$, respectively, for all $n\in S^{d-1}$, $b\in B_i$, and $w\in W_j$.
\end{theorem}

\begin{proof}
The proof goes as in the case of compact support. We now choose $D$ such that $\mu(D) < \min\{\beta_i,1-\omega_j\mid i\in I, j\in J\}$ and all $B_i,W_j \subseteq B(D)$ to ensure that $\supp f^X_ag_a^X\subseteq \partial X \oplus aB(2D)$. 
The same argument then reduces the proof to a computation of the limit as $a\to \infty$ of
\begin{equation}\label{conv?}
 \int_{-2D}^{2D} f^X_a(x+atn) g_a^X (x+atn) dt
\end{equation}
for each $x\in \partial X$.

We still have the inequalities 
\begin{equation*}
f^{\R^d\backslash B_{out}}_a g^{B_{in}}_a \leq f^X_ag_a^X \leq f^{B_{in}}_ag^{\R^d\backslash B_{out}}_a.
\end{equation*}
However, $\theta^{B_{in}}_a$ and $\theta_a^{\R^d\backslash B_{out}}$ may not be injective. It is still true that
\begin{align}
\begin{split}\label{ovre}
\MoveEqLeft \int_{-2D}^{2D} f^{B_{in}}_a(x+atn) g^{\R^d\backslash B_{out}}_a (x+atn)dt \\
&\leq  \big(\min_{i\in I}\{t^{\R^d\backslash B_{out}}_-(a,\beta_i;B_i)\}-\max_{j\in J}\{t_+^{B_{in}}(a,\omega_j;W_j)\}\big)^+.
\end{split}
\end{align}
Moreover, Lemma \ref{tsqueeze} yields
\begin{equation*}
f^{H}_a(x+a(t-\gamma(a))n)g^{H}_a(x+a(t+\gamma(a))n) \leq f^{\R^d\backslash B_{out}}_a (x+atn) g^{B_{in}}_a(x+atn)
\end{equation*}
and hence
\begin{align}
\begin{split}\label{nedre}
\MoveEqLeft \big( \min_{i\in I}\{t^H_-(0,\beta_i;B_i)- \gamma(a)\}-\max_{j\in J}\{t^H_+(0,\omega_j;W_j)+\gamma(a)\}\big)^+\\ 
 &\leq \int_{-2D}^{2D} f^{\R^d\backslash B_{out}}_a (x+atn) g^{B_{in}}_a(x+atn)dt.
 \end{split}
\end{align}
Both the right hand side of \eqref{ovre} and the left hand side of \eqref{nedre} converge to
\begin{align*}
\big(\min_{i\in I}\{t^H_-(0,\beta_i;B_i) \}-\max_{j\in J}\{t^H_+(0,\omega_j;W_j)\}\big)^+
\end{align*}
by Lemma \ref{tsqueeze}. Thus \eqref{conv?} is forced to converge with the same limit.
\end{proof}

\begin{remark}\label{rem}
The condition that $\beta $ is a regular value is easily satisfied given Pro\-per\-ty \ref{rhocond}, e.g.\ if $\rho>0$ almost everywhere.
\end{remark}

\section{Applications to surface area estimation}\label{app1}
The formulas of Section \ref{1storder} have implications to surface area estimation. We derive some of these below.
\subsection{Thresholded images}\label{thri}
Theorem \ref{first1} and \ref{first2} apply directly to local algorithms for thresholded images:
\begin{corollary}\label{threscor}
Let $X\subseteq \R^d$ be compact gentle, $\beta\in [0,1)$, $\rho$ a PSF and $(B_l,W_l)$ a configuration. Suppose $B_i=B_l$, $W_j=W_l$, $\beta_i=\omega_j=\beta$, and $\rho_i=\rho_j=\rho$ satisfy the conditions of either Theorem \ref{first1} or \ref{first2}. Then
\begin{align}\nonumber
\lim_{a\to 0} a^{d-1}EN(\beta)_l^{a\La_c}(X)&=\lim_{a\to 0} a^{-1} \int_{\R^d} \mathds{1}_{\big\{\theta_a^X(z+aB_l)\subseteq (\beta,1]\big\}} \mathds{1}_{\big\{ \theta_a^X(z+aW_l)\subseteq [0,\beta]\big\}}dz \\ \label{th=bw}
&= \int_{\partial X} (-h(B_l\oplus \check{W}_l,n))^+d\Ha^{d-1}\\\nonumber
&= \lim_{a\to 0} a^{d-1} EN_l^{a\La_c}(X).
\end{align}
 In particular, if $\hat{V}_{d-1}$ is a local estimator for black-and-white images,
\begin{equation*}
\lim_{a\to 0} E\hat{V}(\beta)_{d-1}^{a\La_c}(X) =\lim_{a\to 0} E\hat{V}_{d-1}^{a\La_c}(X).
\end{equation*}
\end{corollary}
The last equality in \eqref{th=bw} is \cite[Theorem 5]{rataj}.

The asymptotic mean of surface area estimators applied to thresholded images is thus the same as for  black-and-white images, so the asymptotic results in \cite{ kampf2, am3, johanna} carry over:
\begin{corollary}
Let $d>1$ and let $\hat{V}_{d-1}$ a local algorithm. Let $\rho$ and  $\beta $ be as in Corollary \ref{threscor}. Then $\hat{V}(\beta)_{d-1}$ is asymptotically biased on both the class of $r$-regular sets (see Definition \ref{defreg} below) and on the class of compact convex polytopes with non-empty interior.
\end{corollary}

\subsection{Surface area estimators with $n=1$}\label{direct}
Consider the number of pixels with threshold value in some interval $I\subseteq (0,1)$
\begin{equation*}
N_{I}^{a\La_c}(X)= | \{z\in a\La_c \mid \theta^X_a(z)\in I \}|
\end{equation*}
where $|\cdot|$ denotes cardinality.  This corresponds to the case $n=1$ in Definition \ref{greyest} with the function $f=\mathds{1}_{I}$. Theorem \ref{first1} and \ref{first2} yield:
\begin{corollary}
Let $X\subseteq \R^d$ be a compact gentle set,  $(\beta,\omega] \subseteq (0,1)$, and $\rho$ a PSF. Suppose $\beta_i=\beta$, $\omega_i=\omega$, $B_i=W_j=\{0\}$, and $\rho_i=\rho_j = \rho$ satisfy the conditions of either Theorem \ref{first1} or \ref{first2}. Then
\begin{align}
\lim_{a\to 0}a^{d-1}EN_{(\beta,\omega]}^{a\La_c}(X)
\label{indic}
& = \int_{ \partial X }(\varphi^{\rho}(\beta,n) - \varphi^{\rho}(\omega,n))d\Ha^{d-1}.
\end{align}
In particular, if $\rho$ is rotation invariant, $\varphi^\rho(\beta):=\varphi^{\rho}(\beta,n)$ is independent of $n\in S^{d-1}$, so
\begin{equation}\label{countbw}
\tfrac{1}{2}(\varphi^\rho(\beta)-\varphi^\rho(\omega))^{-1}N_{(\beta,\omega]}^{a\La_c}
\end{equation}
is an asymptotically unbiased estimator for $V_{d-1}$ on the class of compact gentle sets.
\end{corollary} 

\begin{remark}
If one or more of the open and closed ends of $(\beta,\omega]$ are changed, the corresponding $\varphi$ should be replaced by $\tilde{\varphi}$ in \eqref{indic}, yielding a similar statement for all $I\subseteq (0,1)$.

For $I = (0,1)$ and $\rho = \rho_B$ where $B$ is the closure of its interior, this implies
\begin{align*}
\lim_{a\to 0}a^{d-1}EN_{(0,1)}^{a\La_c}(X)
 = \int_{ \partial X }h(B\oplus \check{B}, n)d\Ha^{d-1}.
\end{align*}
In particular, if $X$ and $B$ are convex, this is the mixed volume $2V(X[d-1],B\oplus \check{B})$, see \cite[Section 5]{schneider}. 
\end{remark}

\begin{remark}
Even if the grey-values in the output data are grouped into finitely many (at least three) intervals, an estimator of the form $N_I$ can still be applied.
\end{remark}

Suppose $\rho $ is as in Theorem \ref{first1}. The limit in \eqref{indic} can also be written as 
\begin{equation*}
\int_{\partial X} (\varphi^\rho(\beta,n)-\varphi^\rho (\omega,n))d\Ha^{d-1}=\int_{\partial X} \mu_n((\beta,\omega])d\Ha^{d-1} =\int_{\partial X}\int_{(0,1)}\mathds{1}_{(\beta,\omega]}d\mu_n d\Ha^{d-1}
\end{equation*}
where $\mu_n$ for $n\in S^{d-1}$ is the Lebesgue--Stieltjes measure defined by the increasing right continuous function $\beta \mapsto -\varphi^\rho(\beta,n)$. 

\begin{lemma}
For any Borel set $A\subseteq (0,1)$, the function $S^{d-1} \to \R$ that is given by $n \mapsto \mu_n(A)$ is Borel measurable. In particular, for any compact gentle set $ X\subseteq \R^d$, there is a Borel measure $\mu^X$ on $(0,1)$ defined by
\begin{equation*}
\mu^X(A)=\int_{\partial X}\int_{(0,1)}\mathds{1}_{A}d\mu_n d\Ha^{d-1}.
\end{equation*}
\end{lemma}

\begin{proof}
Since $\varphi^\rho: S^{d-1}\times (0,1) \to \R$ is measurable, $n\mapsto\mu_n(A)$ is clearly measurable for $A$ belonging to the intersection stable collection of $(\beta,\omega]$ for $\beta,\omega \in [0,1)$. The claim now follows from Dynkin's lemma.
\end{proof}

Introduce the image measure 
\begin{equation*}
\mu_a^X = a^{-1}\Ha^{d-1} \circ (\theta^{X}_a)^{-1}
\end{equation*}
on $(0,1)$. 
If $f:(0,1)\to \R$ is bounded measurable,  
\begin{equation*}
E\hat{V}(f)_{d-1}^{a\La_c}(X) = \int_{(0,1)} f d\mu_a^X.
\end{equation*}
Similarly, by standard arguments,
\begin{equation*}
\int_{(0,1)} f d\mu^X = \int_{\partial X}\int_{(0,1)}f d\mu_n d\Ha^{d-1}.
\end{equation*}

Theorem \ref{first1} and \ref{first2} yield:
\begin{corollary}\label{mukonv}
If $\rho$ is as in Theorem \ref{first1}, $\mu_a^X$ converges weakly to $\mu^X$. In particular, for any $f:(0,1)\to \R$ that is  bounded measurable and $\mu^X$-almost everywhere continuous,
\begin{equation}\label{weak}
\lim_{a\to 0} E\hat{V}(f)_{d-1}^{a\La_c}(X) = \int_{(0,1)}f d\mu^X.
\end{equation}

If $\rho$ satisfies Property \ref{rhocond} and $[\beta,\omega] \subseteq (0,1)$ contains only regular values for $\theta_0^{H_n}(\cdot;0)$ for all $n\in S^{d-1}$, the restriction of $\mu_a^X$ to  $(\beta,\omega)$ converges weakly to $\mu^X$ restricted to the same interval. In particular, \eqref{weak} holds in this situation as well if $\supp f \subseteq (\beta,\omega)$.
\end{corollary}

\begin{proof}
Suppose first that $\rho$ is as in Theorem \ref{first1}. Taking $f=\mathds{1}_{(0,\omega]}$ shows that
\begin{equation*}
\lim_{a\to 0} \mu_a^X((0,\omega]) = \mu^X((0,\omega])
\end{equation*}
for all $\omega\in (0,1)$. Moreover, Remark \ref{tilderem} shows that
\begin{equation*}
\lim_{a\to 0} \mu_a^X((0,1)) = \int_{\partial X} (\varphi^\rho(0,n)-\tilde{\varphi}^\rho(1,n))d\Ha^{d-1}.
\end{equation*}
By monotone convergence, this equals
\begin{align*}
 \mu^X((0,1))&=\sup_{\omega<1} \mu^X((0,\omega])\\
 &=\sup_{\omega < 1}\int_{\partial X} (\varphi^\rho(0,n)-\varphi^\rho(\omega,n))d\Ha^{d-1}\\
 &=\int_{\partial X} (\varphi^\rho(0,n)-\inf_{\omega < 1}\{\varphi^\rho(\omega,n)\})d\Ha^{d-1}\\
 &= \int_{\partial X} (\varphi^\rho(0,n)-\tilde{\varphi}^\rho(1,n))d\Ha^{d-1}.
\end{align*}
The weak convergence follows. The non-compact case is similar. 
\end{proof}

If $\rho $ is rotation invariant, $\varphi^\rho(\beta)=\varphi^\rho(\beta,n)$, and hence $\mu:=\mu_n$, is independent of $n\in S^{d-1}$.  Thus \eqref{weak} reduces to:
\begin{corollary} \label{unbint}
Suppose $\rho $ is rotation invariant. Under the assumptions of Corollary~\ref{mukonv},
\begin{equation*}
\lim_{a\to 0} E\hat{V}(f)_{d-1}^{a\La_c}(X) = 2V_{d-1}(X) \int_{(0,1)} f d\mu.
\end{equation*}
That is, $\hat{V}(f)_{d-1}$ is asymptotically unbiased if and only if $2\int_{(0,1)} f d\mu=1$.
\end{corollary} 

If $\rho$ is not rotation invariant, we can get bounds on the worst case asymptotic relative mean error instead:
\begin{corollary}\label{asworst}
Under the assumptions of Corollary \ref{mukonv}, 
\begin{equation}\label{worstcase}
\Err(\hat{V}(f)_{d-1})\leq \sup_{n\in S^{d-1}}\bigg|2\int_0^1 f d\mu_n-{1}\bigg|
\end{equation}
with equality if $\rho$ is reflection invariant or $f$ satisfies $f(x)=f(1-x)$. 

For any $f$, the function $\tilde{f}(x)=\frac{1}{2}(f(x)+f(1-x))$  satisfies $\tilde{f}(x)=\tilde{f}(1-x)$ and 
\begin{equation*}
\Err (\hat{V}(\tilde{f})_{d-1}) \leq \Err (\hat{V}(f)_{d-1}).
\end{equation*}
\end{corollary}

\begin{proof}
$\Err (\hat{V}(f)_{d-1})$ is given by
\begin{align*}
\sup_{X\in \mathcal{S}}\bigg|\frac{\lim_{a\to 0}E\hat{V}(f)_{d-1}^{a\La_c}(X)}{V_{d-1}(X)}-1\bigg| &= \sup_{X\in \mathcal{S}}\bigg|V_{d-1}(X)^{-1}\int_{\partial X}\int_0^1 f d\mu_n d\Ha^{d-1}-1\bigg|\\
&\leq \sup_{X\in \mathcal{S}}\bigg|V_{d-1}(X)^{-1}\int_{\partial X}\bigg|\int_0^1 f d\mu_n -\frac{1}{2} \bigg|d\Ha^{d-1}\bigg| \\
&\leq \sup_{n\in S^{d-1}}\bigg|2\int_0^1 f d\mu_n-{1}\bigg|.
\end{align*}

Let $n_k\in S^{d-1}$ be a sequence with $\big|2\int_0^1 f d\mu_{n_k}-{1} \big|$ converging to the latter supremum and choose an orthonormal basis $u^1_k,\dots , u^{d-1}_k$ for $n_k^\perp$ and a sequence $t_k>0$ such that $\lim_{k\to \infty}{t_k}=0$. Observe that $\mu_n(A)=\mu_{-n}(1-A)$. Assuming $\rho $ reflection invariant or $f(x)=f(1-x)$, the asymptotic relative bias on the sets $[0,t_k n_k] \oplus \bigoplus_{i=1}^{d-1} [0, u^i_k]$ thus converges to the right hand side of the inequality.

The last claim follows from
\begin{align*}
\MoveEqLeft \bigg|\frac{\lim_{a\to 0}E\hat{V}(\tilde{f})_{d-1}^{a\La_c}(X)}{V_{d-1}(X)}-1\bigg|\\
&\leq \frac{1}{2}\bigg(\bigg|\frac{\lim_{a\to 0}E\hat{V}(f)_{d-1}^{a\La_c}(X)}{V_{d-1}(X)}-1\bigg|+\bigg|\frac{\lim_{a\to 0}E\hat{V}(f)_{d-1}^{a\La_c}(-X)}{V_{d-1}(-X)}-1\bigg|\bigg).
\end{align*}
\end{proof}

%

\section{Second order formulas}\label{sof}
To obtain a second order version of Theorem \ref{first1}, we need to be able to control the second order behaviour of the boundary of underlying set $X\subseteq \R^d$. We assume throughout the section that $d>1$. Thus we shall restrict attention to the class of $r$-regular sets:
\begin{definition}\label{defreg}
A closed subset $X\subseteq \R^d $ is called $r$-regular for some $r>0$ if for all $x\in \partial X$ there exist two balls $B_{in}$ and $B_{out}$ of radius $r$ both containing $x$ such that $B_{in}\subseteq X$ and $\indre (B_{out})\subseteq \R^d\backslash X$. The unique outward pointing normal vector at $x$ is denoted by $n(x)$.
\end{definition}
It can be proved \cite{federer}  that if $X$ is $r$-regular, then $\partial X$ is a $C^1$ manifold with $\Ha^{d-1}$-almost everywhere dif\-fe\-ren\-tia\-ble normal vector field. In particular, its principal curvatures $k_1,\dots, k_{d-1} \leq r^{-1}$, corresponding to the orthogonal principal directions $e_1,\dots,e_{d-1}\in T\partial X$, can be defined almost everywhere  as the eigenvalues of the differential $dn$. 
Thus the second fundamental form $\II_x$ on the tangent space $T_x\partial X$  is defined for $\Ha^{d-1}$-almost all $x\in \partial X$. For $\sum_{i=1}^{d-1}\alpha_ie_i\in T_x\partial X$, $\II_x$ is the quadratic form given by
\begin{equation*}
\II_x\left(\sum_{i=1}^{d-1}\alpha_ie_i \right)= \sum_{i=1}^{d-1}k_i(x)\alpha_i^2
\end{equation*}
whenever $d_xn$ is defined. In particular, the trace is $\tr  \II =k_1+\dotsm+k_{d-1}$.

The integrated mean curvature $2\pi(d-1)^{-1} V_{d-2}$ can thus be defined \cite{federer} for $r$-regular sets by 
\begin{equation*}
V_{d-2}(X)=\frac{1}{2\pi} \int_{\partial X} \tr(\II) d\Ha^{d-1}.
\end{equation*}

Let $T^\eps \partial X =\{(x,\alpha)\mid \alpha\in T_x\partial X, |\alpha|<\eps \}$.
We need the following lemma. A proof can be found e.g.\ in \cite{am2}.
\begin{lemma}\label{boundary}
Let $X$ be an $r$-regular set.
There is a unique function $q: T^r \partial X \to \R$ such that for $\alpha \in T_x\partial X$,  $q(x,\alpha)$ is the unique $q\in [-r,r]$ with $x+\alpha + qn(x)\in \partial X$. There is a constant $C>0$ such that
\begin{equation*}
q(x,\alpha)\leq C|\alpha|^2
\end{equation*}
for all $(x,\alpha)\in T ^r\partial X$. Moreover,
\begin{equation*}
\lim_{a\to 0} a^{-2}q(x,a\alpha) = -\tfrac{1}{2}\II_x(\alpha).
\end{equation*}
\end{lemma}

We also make the following observation:
\begin{lemma}\label{decrease}
If $X$ is $r$-regular, $K>0$, and $\rho$ has compact support, then for all $a$ sufficiently small, the map $t \mapsto \theta^X_a(t;s)$ is decreasing on the interval $[-a^{-1}r,a^{-1}r]$ for all $s\in B(K)$.
\end{lemma}

\begin{proof}
Suppose $\supp \rho \subseteq B(D)$.
 By Lemma \ref{boundary},
\begin{equation*}
(x+a(tn + s + B(D)))\cap X -a\nu n \subseteq (x+a((t-\nu)n + s + B(D))) \cap X
\end{equation*}
for all $s\in B(K)$, $\nu>0$, and $t,t-\nu \in [-a^{-1}r,a^{-1}r]$ whenever $a$ is sufficiently small. Hence,
\begin{align*}
\theta_a^X(t;s) &= \int_{(x+a(tn + s + B(D)))\cap X -\nu n}\rho_a(z-(x+a((t-\nu)n + s) )) dz \\
&\leq  \int_{(x+a((t-\nu)n + s + B(D)))\cap X}\rho_a(z-(x+a((t-\nu)n + s))) dz\\
&=\theta_a^X(t-\nu;s).
\end{align*}
\end{proof}

For $z\in \R^d$ and $n\in S^{d-1}$, we write $z=(z_{n^\perp},z_n)$ where $z_n=\langle z,n \rangle \in \R$ and $z_{n^\perp}\in n^\perp$ is the projection of $ z$ onto $n^\perp$. 

Let $x\in \partial X$. Define the quadratic approximation $Q(x)$ to $X$ at $x$ by
\begin{equation*}
Q(x)=\{z\in \R^d \mid |(z-x)_{n}| \leq -\tfrac{1}{2}\II_x((z-x)_{n^\perp}) \}.
\end{equation*}
If $x\in \partial X$ is understood, we simply write $Q:=Q(x)$.


\begin{definition}
Suppose $\rho$ is continuous with compact support.
For $x\in \partial X$ and $\beta_0$ a regular value for $\theta^H_0(\cdot;s)$, define
\begin{equation*}
\psi^{Q(x)}(\beta_0;s) = \frac{\psi_1\big(t^{H}(0,\beta_0;s);s\big)}{\psi_2\big(t^{H}(0,\beta_0;s);s\big)}=- \psi_1\big(t^{H}(0,\beta_0;s);s\big)\frac{d}{d\beta}t^{H}(0,\beta;s)_{\mid \beta = \beta_0}
\end{equation*}
where
\begin{align*}
\psi_1(t;s)&=-\frac{1}{2}\int_{n(x)^\perp}  \II_x(z)  \rho(z-s_{n^\perp}, -t-s_n)) dz\\
\psi_2(t;s)&=-\frac{d}{dt}\theta^H_0(t;s)= \int_{n(x)^\perp}   \rho(z-s_{n^\perp}, -t-s_n)) dz.
\end{align*}
\end{definition}

\begin{lemma}\label{tHtQ}
Let $X$ be $r$-regular and $x\in \partial X$. Suppose $\rho $ is continuous and has compact support. Let $B\subseteq \R^d$ compact and assume $\beta_0\in (0,1)$ is a regular value for $\theta^H_0(\cdot ;s)$ for all $s\in B$. Then the function $(a,s) \mapsto t^Q(a,\beta_0;s)$ extends to a well-defined $C^1$ function on $(-\eps, \eps) \times U $ for some $\eps >0 $ and $U\supseteq B$ open so that for all $(a,s) \in (0, \eps) \times U $, $t^Q(a,\beta_0;s)$ is the unique $t$ with $\theta^{Q}_a(t;s)=\beta_0$. Moreover,
\begin{equation*}
t^Q(a,\beta_0;s)= t^H(0,\beta_0;s)+ a\psi^{Q}(\beta_0;s) + o(a)
\end{equation*}
and $\sup_{s\in B}|t^Q(a,\beta_0;s)- t^H(a,\beta_0;s)|\in O(a)$.
\end{lemma}

\begin{proof}
First observe that
\begin{align*}
\theta_a^Q(t;s)= \theta^H_a(t;s) + \theta^{Q\backslash H}_a( t;s) - \theta^{H \backslash Q}_a( t;s). 
\end{align*}
Suppose $\supp \rho , B \subseteq B(D)$. We first rewrite the latter terms:
\begin{align*}
\theta_a^{Q\backslash H}( t;s)& =\int_{Q\backslash H } \rho_a(z-(x+a(tn+s))) dz\\
&= a^{-d}\int_{n^\perp}\int_{0}^{0\vee \big(-\tfrac{1}{2}\II(z_{n^\perp})\big)} \rho(a^{-1}z_{n^\perp}-s_{n^\perp},a^{-1}z_n-( t+s_n)) dz_{n} dz_{n^\perp}\\
&= a^{-1}\int_{B^{n^\perp}(2D)}\int_{0}^{0\vee \big(-a^{2}\tfrac{1}{2}\II(z_{n^\perp})\big)} \rho(z_{n^\perp}-s_{n^\perp},a^{-1}z_n-( t+s_n)) dz_n dz_{n^\perp}\\
&= a \int_{B^{n^\perp}(2D)}\int_{0}^{0\vee \big(-\tfrac{1}{2}\II(z_{n^\perp})\big)} \rho(z_{n^\perp}-s_{n^\perp},az_n - (t+s_n)) dz_{n} dz_{n^\perp}
\end{align*}
where $B^{n^\perp}(2D)$ is the ball in $n^{\perp}$ of radius $2D$. Similarly,
\begin{align*}
\theta^{H\backslash Q}_a( t;s)= a \int_{B^{n^\perp}(2D)}\int_{0\wedge \big(-\tfrac{1}{2}\II(z_{n^\perp})\big)}^{0} \rho(z_{n^\perp}-s_{n^\perp},az_n- (t+s_n)) dz_n dz_{n^\perp}.
\end{align*}
This computation shows that $\theta^{Q}_a(t;s)$ extends continuously to a well-defined function for all $(a,t,s)\in \R^{2+d}$.
Denote this function by 
\begin{equation*}
\beta(a,t;s)= \theta^H_0(t;s)+ a \int_{B^{n^\perp}(2D)}\int_{0}^{-\tfrac{1}{2} \II(z_{n^\perp})} \rho(z_{n^\perp}-s_{n^\perp}, az_n-( t+s_n)) dz_n dz_{n^\perp}.
\end{equation*}
The assumptions on $\rho$ imply that $\beta(a,t;s)$ is $C^1$ in $(a,t,s)$ and
\begin{align*}
\frac{d}{da}\beta(a,t;s)={}&  -\frac{1}{2}\int_{B^{n^\perp}(2D)}  \II(z_{n^\perp})\rho(z_{n^\perp}-s_{n^\perp}, -a\tfrac{1}{2}\II(z_{n^\perp})-( t+s_n )) dz_{n^\perp}\\
\frac{d}{dt}\beta(a,t;s)={} & \frac{d}{dt}\theta^H_0(t;s) + a\int_{B^{n^\perp}(2D)}\Big( \rho(z_{n^\perp}-s_{n^\perp}, -(t+s_n))\\
&-\rho(z_{n^\perp}-s_{n^\perp},- a \tfrac{1}{2}\II(z_{n^\perp}) -(t+s_n))\Big)dz_{n^\perp}.
\end{align*}
Again, these functions are clearly continuous.

In particular, at $a=0$ we obtain
\begin{align*}
\beta(0,t;s)&= \theta^H_0(t;s)\\
\frac{d}{da}\beta(a,t;s)_{\mid a=0}& = -\frac{1}{2}\int_{n^\perp} \II(z_{n^\perp})
\rho(z_{n^\perp}-s_{n^\perp}, -(t+s_n)) dz_{n^\perp} \\
\frac{d}{dt}\beta(a,t;s)_{\mid a=0} &= \frac{d}{dt}\theta^H_0(t;s).
\end{align*}

Since $\frac{d}{dt}\beta(0,t;s)_{\mid t=t^H(0,\beta_0;s)} < 0$ for all $s\in B$ by assumption, the implicit function theorem yields that in a neighborhood of the compact set $\{0\}\times B$, the solution $t$ to $\beta(a,t;s)=\beta_0$ is given by a $C^1$ function $(a,s)\mapsto t^Q(a,\beta_0;s)$ with 
\begin{equation*}
t^Q(a,\beta_0;s) = t^H(0,\beta_0;s) - a\frac{\frac{d}{da} \beta(a,t^H(0,\beta_0;s);s )_{\mid a=0}}{\frac{d}{dt}\beta(0,t;s)_{\mid t=t^H(0,\beta_0;s)}} + o(a).
\end{equation*}

The last claim follows from the mean value theorem, since for $0\leq a_0 \leq \frac{\eps}{2}$,
\begin{align*}
|t^Q(a_0,\beta_0;s)- t^H(0,\beta_0;s)| &=a_0 \Big|\frac{d}{da}t^Q(a,\beta_0;s)_{\mid a=a'}\Big|\\
&\leq a_0\sup \Big\{ \Big|\frac{d}{da}t^Q(a,\beta_0;s)\Big|\, \Big| a \in \big[0,\tfrac{\eps}{2}\big], s \in B \Big\}
\end{align*}
where  $a'\in [0,a_0]$ and the latter supremum is finite by continuity.
\end{proof}

\begin{lemma}\label{tXtQ}
Let $X$ be $r$-regular, $x\in \partial X$, and $B\subseteq \R^d$ compact.
Let $\rho$ be continuous with compact support.  Then there is a function $\lambda(a)\in o(a)$ such that 
\begin{equation*}
|\theta^{Q}_a(t;s)-\theta^X_a(t;s)| \leq \lambda(a)
\end{equation*}
for all $t\in [-a^{-1}r,a^{-1}r]$ and $s\in B$.

If, moreover, $\beta_0$ is a regular value for $\theta^{H}_0(\cdot;s)$ for all $s\in B$, there is a constant $M>0$  such that for all $s\in B$,
\begin{equation*}
|t^{Q}(a,\beta_0;s)-t^X(a,\beta_0;s)| \leq M\lambda(a).
\end{equation*}
\end{lemma}

\begin{proof}
Suppose $B,\supp \rho \subseteq B(D)$. Observe that
\begin{align*}
\MoveEqLeft |\theta^{Q}_a(t;s)-\theta^X_a(t;s)| \leq \int_{Q\backslash X\cup X\backslash Q} \rho_a(z-(x+a(tn+s)))dz \\
&= a\int_{n^\perp}\int_{\big(-\tfrac{1}{2}\II(z_{n^\perp})\big)\wedge a^{-2}q(x,az_{n^\perp}) }^{\big(-\tfrac{1}{2} \II(z_{n^\perp}) \big)\vee a^{-2}q(x,az_{n^\perp})} \rho(z_{n^\perp}-s_{n^\perp},az_n-(t+s_n)) dz_n dz_{n^\perp}\\
&\leq a\sup \rho \int_{B^{n^\perp}(2D)}|{\tfrac{1}{2}\II(z_{n^\perp})}+{a^{-2}q(x,az_{n^\perp})}|  dz_{n^\perp} .
\end{align*}
As $|{\tfrac{1}{2}\II(z_{n^\perp})}+{a^{-2}q(x,az_{n^\perp})}|$ is bounded for $z_{n^\perp}\in T_x^{2D} \partial X$ and $a\in (0,\frac{r}{2D}]$ by Lemma~\ref{boundary}, the same lemma combined with Lebesgue's theorem yields that
\begin{equation*}
\lambda(a)=a\sup \rho  \int_{B^{n^\perp}(2D)}|{\tfrac{1}{2}\II(z_{n^\perp})}+{a^{-2}q(x,az_{n^\perp})}|  dz_{n^\perp} \in o(a).
\end{equation*}

Now suppose  $ \beta_0$ is a regular value for $\theta^H_0(\cdot;s)$ for all $s\in B$.
The function $\frac{d}{dt} \beta(a,t;s)$ from the proof of Lemma~\ref{tHtQ} was continuous in $(a,t,s)$, so there is a neighborhood of the compact set $\{ (0,t^H(0,\beta_0;s),s)\in \R^{2+d} \mid s\in B \}$ on which $\frac{d}{dt} \beta(a,t;s)>0$.  In particular, there are constants $ \delta, \eps, M_1>0$ such that
\begin{equation*}
-\inf\Big\{\frac{d}{dt} \beta(a,t;s) \, \Big|\, a\in [0,\eps],\, s\in B,\, |t-t^H(0,\beta_0;s)|\leq \delta \Big\}=M_1.
\end{equation*}
Thus for $a\in (0,\eps)$ and $t,t+\nu \in [t^H(0,\beta_0;s)-\delta,t^H(0,\beta_0;s)+\delta]$,
\begin{equation*}
\theta^{Q}_a(t+\nu;s) - \theta^{Q}_a(t;s) \leq -M_1 \nu.
\end{equation*}
Hence,  
\begin{equation*}
\theta^{X}_a(t+\nu;s) - \theta^{Q}_a(t;s) \leq \lambda(a)-M_1 \nu.
\end{equation*}
As $\lim_{a\to 0} t^{Q}(a,\beta_0;s)= t^{H}(0,\beta_0;s)$ uniformly for $s \in B$ by Lemma \ref{tHtQ},
\begin{equation*}
|t^{Q}(a,\beta_0;s)-t^H(0,\beta_0;s)|<\tfrac{1}{2}\delta
 \end{equation*}
for all $s\in B$ and $a$ sufficiently small.
Thus, if $\theta^{Q}_a(t;s)=\beta_0$, then $\theta^{X}_a(t+\nu;s)< \beta_0$ for $\tfrac{1}{2}\delta \geq \nu > M_1^{-1}\lambda(a)$. So if $a$ is so small that $\tfrac{1}{2}\delta  > M_1^{-1}\lambda(a)$,
\begin{equation*}
t^{X}(a,\beta_0;s)-t^{Q}(a,\beta_0;s)\leq M_1^{-1}\lambda(a)
\end{equation*}
for all $s\in B$.
The other inequality is similar.
\end{proof}

\begin{theorem}\label{second}
Let $X$ be a closed $r$-regular set and $A\subseteq \R^d$ be bounded measurable. Let $I$ and $J$ be non-empty finite index sets. For $i\in I$ and $j\in J$, let $B_i,W_j\subseteq \R^d$ be non-empty compact strictly convex sets and let $\rho_i,\rho_j$ be continuous PSF's with compact support. Suppose that  $ \beta_i ,\omega_j \in (0,1)$ are regular  values for $\theta_0^{H_{n},\rho_i}(\cdot ;b)$ and $\theta_0^{H_{n},\rho_j}(\cdot;w)$, respectively, for all $n\in S^{d-1}$, $b\in B_i$, and $w\in W_j$. Then 
\begin{align*}
& \int_{\xi^{-1}_{\partial X}(A)} \prod_{i\in I}\mathds{1}_{\big\{\theta_a^{X,\rho_i}(x+aB_i) \subseteq (\beta_i, 1] \big\} }\prod_{j\in J} \mathds{1}_{\big\{\theta_a^{X,\rho_j}(x+aW_j) \subseteq [0,\omega_j]\big\} }dx \\
&= a\int_{ \partial X \cap A} \big(t^{H}_-(0,\beta;B) - t^{H}_+(0,\omega;W) \big)^+ d\Ha^{d-1} \\
&\quad + a^{2}\int_{ \partial X\cap A} \bigg(\tfrac{1}{2}\tr \II \big(t^{H}_-(0,\beta;B)^2- t^{H}_+(0,\omega;W)^2 \big)\mathds{1}_{ \big\{t^{H}_-(0,\beta;B)>t^{H}_+(0,\omega;W)\big\}} \\
&\quad +\big(\min_{i\in I'(n)}\{\psi^{Q,\rho_i}(\beta_i,B_i)\}-\max_{j\in J'(n)}\{\psi^{Q,\rho_j}(\omega_j;W_j\big)\}) \mathds{1}_{ \big\{t^{H}_-(0,\beta;B)>t^{H}_+(0,\omega;W)\big\}}\\
&\quad + \big(\min_{i\in I'(n)}\{\psi^{Q,\rho_i}(\beta_i,B_i)\}-\max_{j\in J'(n)}\{\psi^{Q,\rho_j}(\omega_j;W_j)\}\big)^+\mathds{1}_{ \big\{t^{H}_-(0,\beta;B)=t^{H}_+(0,\omega;W)\big\} }\bigg)d\Ha^{d-1} \\
&\quad +o(a^2).
\end{align*}
\end{theorem}

 The following notation is used in the theorem and its proof:
\begin{align*}
t^X_-(a,\beta;B)&=\min_{i\in I} \{t^X_-(a,\beta_i;B_i)\}\\
t^X_+(a,\omega;W)&=\max_{j\in J} \{t^X_+(a,\omega_j;W_j)\}.
\end{align*}
Moreover, $I',J'$ are the index sets
\begin{align*}
I'(n)&=\{i_0 \in I \mid t^{H_{n}}_-(0,\beta;B)=t^{H_{n}}_-(0,\beta_{i_0};B_{i_0})\}\\
J'(n)&=\{j_0\in J \mid t^{H_{n}}_+(0,\omega;W)=t^{H_{n}}_+(0,\omega_{j_0};W_{j_0})\}
\end{align*}
and 
\begin{align*}
\psi^{Q(x),\rho_i}(\beta_i;B_i)&=\psi^{Q(x),\rho_i}(\beta_i;b_i(n))\\
\psi^{Q(x),\rho_j}(\omega_j;W_j)&=\psi^{Q(x),\rho_j}(\omega_j;w_j(n))
\end{align*}
where $b_i(n)\in B_i$ and $w_j(n)\in W_j$ are unique with $h(B_i,n)=\langle b_i(n),n \rangle$ and $h(\check{W}_j,n)=-\langle w_j(n),n \rangle$, respectively.
 
\begin{proof}
For an $r$-regular set $X$, the formula \eqref{lastform} simplifies for $2aD<r$ to the Weyl tube formula
\begin{align*}
\MoveEqLeft\int_{\R^d}\mathds{1}_{\xi^{-1}_{\partial X}(A)} f_a^X(x)g_a^X(x)dx\\
& =\sum_{m=1}^d  \int_{\partial X \cap A} \int_{-r}^{r} t^{m-1} f_a^X(x+tn)g_a^X(x+tn) dt s_{m-1}(k) \Ha^{d-1}(dx)
\end{align*}
where $D$ is chosen as in the the proof of Theorem \ref{first1} and $s_{m}(k)$ denotes the $m$th symmetric polynomial in the principal curvatures.
Again
\begin{align*}
a^{-2}\int_{-r}^{r} t^{m-1} f_a^X(x+tn)g_a^X(x+tn) dt \leq m^{-1}a^{m-2}(2D)^m
\end{align*}
and hence Lebesgue's theorem yields
\begin{align*}
\lim_{a\to 0}a^{-2}\sum_{m=2}^d \int_{\partial X} \int_{-r}^{r} t^{m-1} f_a^X(x+tn)g_a^X(x+tn) dt s_{k-1}(k) \Ha^{d-1}(dx)\\
=  \int_{\partial X } \bigg( \lim_{a\to 0} \int_{-2D}^{2D} t f_a^X(x+atn)g_a^X(x+atn) dt \bigg)s_1(k) \Ha^{d-1}(dx)
\end{align*}
if the limit of the inner integral exists.
The $m=1$ term will be treated separately.

Again we get the inequalities \eqref{ineq2} for all $a$ small enough 
and thus
\begin{align}\nonumber
\MoveEqLeft \tfrac{1}{2}((t^{B_{in}}_-(a,\beta;B)^+)^2-(t^{\R^d\backslash B_{out}}_+(a,\omega;W)^+)^2 )\mathds{1}_{\big\{t^{B_{in}}_-(a,\beta;B)>t^{\R^d\backslash B_{out}}_+(a,\omega;W)\big\}}\\ \nonumber
& = \int_0^{2D} t f_a^{\R^d\backslash B_{out}}(x+atn)g_a^{B_{in}}(x+atn)  dt\\ \label{ulig}
& \leq \int_0^{2D} t f_a^X(x+atn)g_a^X(x+atn) dt \\ \nonumber
& \leq \int_0^{2D} t f_a^{B_{in}}(x+atn)g_a^{\R^d\backslash B_{out}}(x+atn) dt \\\nonumber
&  = \tfrac{1}{2}((t^{\R^d\backslash B_{out}}_-(a,\beta;B)^+)^2- (t^{B_{in}}_+(a,\omega;W)^+)^2)\mathds{1}_{\big \{t^{\R^d\backslash B_{out}}_-(a,\beta;B)> t^{B_{in}}_+(a,\omega;W)\big\}}
\end{align}
so \eqref{unift} forces the middle integral to converge to
\begin{align*}
\tfrac{1}{2}((t^{H}_-(0,\beta;B)^+)^2-(t^{H}_+(0,\omega;W)^+)^2)\mathds{1}_{\{t^{H}_-(0,\beta;B)>t^{H}_+(0,\omega;W)\}}.
\end{align*}
The integration over $[-2D,0]$ is similar, except the inequalities in \eqref{ulig} are switched.

It remains to determine the asymptotics of
\begin{equation}\label{1stterm}
a^{-1} \int_{\partial X\cap A} \int_{-2D}^{2D} \big(f^{X}_a (x+atn) g^{X}_a(x+atn)dt - (t^H_-(0,\beta;B)-t^H_+(0,\omega;W))^+\big)\Ha^{d-1}(dx).
\end{equation}
The proof of Theorem \ref{first1}  yields $M,\eps>0$ depending only on $r$, $\rho$, and $D$ such that
\begin{align*} 
( t^H_-(0,\beta,B) -t^H_+(0,\omega;W)-2Ma)^+
& \leq \int_{-D}^{D} f^{X}_a (x+atn) g^{X}_a(x+atn)dt\\
 & \leq   (t^{H}_-(0,\beta;B)-t_+^{H}(0,\omega;W)+2Ma)^+
\end{align*}
for all $a<\eps$. Hence
\begin{equation}\label{intlim}
a^{-1}\bigg(\int_{-2D}^{2D} f^{X}_a (x+atn) g^{X}_a(x+atn)dt- (t^H_-(0,\beta;B)-t_+^H(0,\omega;W))^+\bigg)\leq 2M,
\end{equation}
allowing us to apply Lebesgue's theorem to \eqref{1stterm}. Since $\theta_a^X(\cdot;s)$ is decreasing by Lemma \ref{decrease},
\begin{equation*}
\int_{-2D}^{2D} f^{X}_a (x+atn) g^{X}_a(x+atn)dt = ( t^X_-(a,\beta;B)-t^X_+(a,\omega;W))^+.
\end{equation*}

By Lemma \ref{tXtQ},
\begin{align*}
\MoveEqLeft \big|\big( t^X_-(a,\beta_i ;B)-t^X_+(a,\omega_j;W)\big)^+ - \big( t^{Q}_-(a,\beta_i;B)-t^{Q}_+(a,\omega_j;W)\big)^+\big|\\
\leq &
\big|t^{Q}_-(a,\beta_i;B)- t^{X}_-(a,\beta_i;B)\big| +  \big|t^{Q}_+(a,\omega_j; W) -t^{X}_+(a,\omega_j;W)\big|\\
\leq &\max_{i\in I}\sup_{b\in B_i}\{| t^{Q}(a,\beta_i;b)- t^{X}(a,\beta_i;b)|\}+\max_{j\in J} \sup_{w\in W_j} \{|t^{Q}(a,\omega_j;w)
-t^{X}(a,\omega_j;w)|\}\\
\leq &2M\lambda(a).
\end{align*}
Hence the limit of \eqref{intlim} equals the limit of 
\begin{align}\label{lim?}
\MoveEqLeft a^{-1}(( t^{Q}_-(a,\beta;B)- t^{Q}_+(a,\omega;W))^+ -( t_-^{H}(0,\beta;B)-t^{H}_+(0,\omega;W))^+ ).
\end{align}
The last part of Lemma \ref{tHtQ} yields that
\begin{align*}
\MoveEqLeft| | t^{Q}_-(a,\beta;B)-t^{Q}_+(a,\omega;W)|-| t^{H}_-(0,\beta;B)-t^{H}_+(0,\omega;W)||\\
&\leq |t^{Q}_-(a,\beta;B)-t^{H}_-(0,\beta;B)|+|t^{Q}_+(a,\omega;W)-t^{H}_+(0,\omega;W)|\\
&\leq 2Ma
\end{align*}
so that \eqref{lim?} equals
\begin{align*}
&a^{-1}\Big((t^{Q}_-(a,\beta;B)- t^{H}_-(0,\beta;B)-(t^{Q}_+(a,\omega;W) -t^{H}_+(0,\omega;W)))\mathds{1}_{ \{t^{H}_-(0,\beta;B)>t^{H}_+(0,\omega;W)\}}\\
&+(t^{Q}_-(a,\beta;B)- t^{H}_-(0,\beta;B)-(t^{Q}_+(a,\omega;W) -t^{H}_+(0,\omega;W)))^+\mathds{1}_{ \{t^{H}_-(0,\beta;B)=t^{H}_+(0,\omega;W)\}}\Big)
\end{align*}
for sufficiently small $a$.

As $B_i$ is strictly convex, there is a unique $b_i\in B_i$ with $h(B_i,n)=\langle b_i,n\rangle$. In particular,
\begin{equation}\label{th}
\min_{i\in I} \inf_{b\in B_i}\{ t^{H}(0,\beta_i;b)\}= \min_{i\in I} t^{H}(0,\beta_i;b_i)=t^{H}(0,\beta_{i_0};b_{i_0})
\end{equation}
for all $i_0\in I'(n)$.
 
Since $b\mapsto t^{Q}(a,\beta_i;b)$ is continuous and $B_i$ is compact, there is a  $b_i(a)\in B_i$ for every $a$ such that
\begin{equation*}
\inf_{b\in B_i}\{ t^{Q}(a,\beta_i;b)\}=t^{Q}(a,\beta_i;b_i(a)).
\end{equation*}
On the other hand, Lemma \ref{tHtQ} yields an $M>0$ such that for all $b\in B_i$,
\begin{equation*}
|t^{Q}(a,\beta_i;b)-t^{H}(0,\beta_i;b)|\leq Ma.
\end{equation*}
Thus $t^{Q}(a,\beta_i;b_i(a))\leq t^{Q}(a,\beta_i;b_i)$ implies that
\begin{equation*}
0\leq t^{H}(0,\beta_i;b_i(a)) - t^{H}(0,\beta_i;b_i) =  \langle b_i, n \rangle - \langle b_i(a), n \rangle \leq 2Ma.
\end{equation*}
Strict convexity and compactness of $B_i$ thus implies that $\lim_{a\to 0} b_i(a)= b_i$ and again continuity yields
\begin{equation*}
\lim_{a\to 0}t^{Q}_-(a;\beta_i;B_i)=\lim_{a\to 0}t^{Q}(a,\beta_i;b_i(a)) = t^{Q}(0,\beta_i;b_i)=t^H_-(0,\beta_i;B_i).
\end{equation*}
Using \eqref{th}, this yields
\begin{equation*}
t^{Q}_-(a,\beta_i;B) - t^H_-(0,\beta_i;B)= \min_{i\in I'(n)}\{ t^{Q}_-(a,\beta_i;B_i) - t^H_-(0,\beta_i;B_i)\}
\end{equation*}
for $a$ sufficiently small.

On the other hand, Lemma \ref{tHtQ} shows that there  are $a',a''\in [0,a]$ such that
\begin{align*}
t^{Q}(a,\beta_i;b_i)-t^{H}(0,\beta_i;b_i)&=a \frac{d}{da}t^{Q}(a,\beta_i;b_i)_{\mid a=a'}\\
t^{Q}(a,\beta_i;b_i(a))-t^{H}(0,\beta_i;b_i(a))&=a\frac{d}{da}t^{Q}(a,\beta_i;b_i(a))_{\mid a=a''}.
\end{align*}
Subtracting these equations and using $t^{H}(0,\beta_i;b_i)\leq t^{H}(0,\beta_i;b_i(a))$ yields
\begin{align*}
0&\leq a^{-1}(t^{Q}(a,\beta_i,b_i)- t^{Q}(a,\beta_i;b_i(a))) \\
&\leq \frac{d}{da}t^{Q}(a,\beta_i;b_i)_{\mid a=a'}-\frac{d}{da}t^{Q}(a,\beta_i;b_i(a))_{\mid a=a''}.
\end{align*}
The right hand side goes to zero for $a\to 0$ by continuity of $(a,b)\mapsto \frac{d}{da}t^{Q}(a,\beta_i;b)$, so 
\begin{align*}
\lim_{a\to 0} a^{-1}(t^{Q}_-(a,\beta_i;B) - t^H_-(0,\beta_i;B)) & =\min_{i\in I'(n)}\{ \lim_{a\to 0} a^{-1}(t^{Q}_-(a,\beta_i;B_i) - t^H_-(0,\beta_i;B_i))\}\\
&=\min_{i\in I'(n)}\{ \lim_{a\to 0} a^{-1}(t^{Q}(a,\beta_i;b_i(a)) - t^H(0,\beta_i ;b_i))\}\\
&= \min_{i\in I'(n)}\{ \lim_{a\to 0} a^{-1}(t^{Q}(a,\beta_i;b_i) - t^H(0,\beta_i ;b_i))\}\\
&=\min_{i\in I'(n)}\{ \psi^{Q,\rho_i}(\beta_i;b_i)\}.
\end{align*}
The $W$ terms in \eqref{lim?} are handled similarly, completing the proof.
\end{proof}

%
%

The next theorem is a modification intended for estimators of the type \eqref{n1}. For $n\in S^{d-1}$, let $\nu_n$ be the signed measure
\begin{equation*}
\nu_n = \nu_n^1-\nu_n^2
\end{equation*}
where $\nu_n^1$ is the Lebesgue--Stieltjes measure on the interval $(0,1)$ defined by the function $\beta \mapsto -\frac{1}{2}(\varphi^\rho(\beta,n)^+)^2$ and $\nu_n^2$  the Lebesgue--Stieltjes measure defined by the function $\beta \mapsto \frac{1}{2}(\varphi^\rho(\beta,n)^-)^2$.

\begin{theorem}\label{fest}
Let $X$ be a compact $r$-regular set.
Let $\rho $ be continuous with compact support such that all $\beta \in (0,1)$ are regular values for $\theta^{H_{n}}_0(\cdot;0)$ for all $n\in S^{d-1}$. Let $f:[0,1]\to \R$ have $\supp f \subseteq [\beta, \omega]$ for some $\beta,\omega \in (0,1)$ and suppose $f$ is $C^1$ on $(\beta,\omega)$ with $f'$ bounded and that $f_+(\beta)=\lim_{x\to \beta^+}f(x)$ and $f_-(\omega)=\lim_{x\to \omega^-}f(x)$ exist. Then 
\begin{align*}
\MoveEqLeft \int_{\R^d} f\circ \theta_a^X d\Ha^d = a\int_{(0,1)} f d\mu_X\\
 &+ a^2\int_{\partial X}\bigg(\tr \II \int_{(0,1)}f d\nu_n - \frac{1}{2} \int_\R f'(\theta_0^{H_{n}}(t;0))\int_{n^\perp} \II(z)\rho(z-tn)dz dt \\
&\quad +  f_+(\beta)\psi^{Q}(\beta;0)-f_-(\omega)\psi^{Q}(\omega;0)\bigg) d\Ha^{d-1}\\
&+o(a^2). 
\end{align*}
\end{theorem}

\begin{proof}
Since $|f \circ \theta_a^X | \leq M\mathds{1}_{ \partial X \oplus aB(D)}$  for some $M>0$ if $\supp \rho \subseteq B(D)$, we still have the formula
\begin{align*}
a^{-2}\int_{\R^d} f\circ \theta_a^X(z) dz
=a^{-2}\sum_{m=1}^d  \int_{\partial X } \int_{-aD}^{aD} t^{m-1} f\circ \theta_a^X (x+tn) dt s_{m-1}(k) \Ha^{d-1}(dx)
\end{align*}
and the same arguments as in the proof of Theorem \ref{second} show that Lebesgue's theorem can be applied to determine the limit of terms with $m\geq 2$ and that all terms with $m\geq 3$ vanish asymptotically.

For a Borel set $A\subseteq (0,1)$, let
\begin{equation*}
\nu_{n,a}^1(A)= \int_{0}^D t \mathds{1}_{A\cap (\beta,\omega)}( \theta_a^X(x+atn) ) dt.
\end{equation*}
This defines a measure concentrated on $(\beta,\omega)$ where it coincides with the Lebesgue-Stieltjes measure determined by the function $\alpha \mapsto -(t^{X}(a,\alpha;0)^+)^2$.
It follows from the proof of Theorem \ref{second} and the same arguments as in the proof of Corollary \ref{mukonv} that $\nu_{n,a}^1$ converges weakly to $\nu_n^1$ and hence
\begin{equation*}
\lim_{a\to 0} \int_{0}^D t f \circ \theta_a^X(x+atn) dt= \lim_{a\to 0} \int_{(0,1)} f d\nu_{n,a}^1 =\int_{(0,1)} f d\nu_n^1.
\end{equation*}
The integration over $[-r,0]$ is handled similarly, showing that  
\begin{equation*}
\lim_{a\to 0} \int_{-D}^D t f \circ \theta_a^X(x+atn) dt=\int_{(0,1)} f d\nu_n.
\end{equation*}

It remains to consider the $m=1$ term.
The assumptions on  $\varphi(\cdot, n)^{-1}= \theta^{H}_0(\cdot;0)$ ensures
\begin{equation*}
\int_{(0,1)} f d\mu_n = \int_{0}^1 f(\beta)  \frac{d}{d\beta} \varphi(\beta,n) d\beta = \int_{-D}^D f \circ \theta^{H}_0(t;0) dt.
\end{equation*} 
Thus we must determine the limit of
\begin{equation*}
a^{-1} \int_{\partial X}\int_{-D}^D( f \circ \theta^X_a(t;0)-f \circ \theta^{H}_0(t;0)) dt d\Ha^{d-1}.
\end{equation*} 
Note that
\begin{equation*}
 |f\circ \theta_a^X(t;0)-f\circ \theta_0^H(t;0)| \leq  \sup|f'|| \theta_a^X(t;0)- \theta_0^H(t;0)|\mathds{1}_{A_1} +  \sup |f| \mathds{1}_{\R \backslash A_1 \cup A_2}
\end{equation*}
where
\begin{align*}
A_1&=\{t\in (-D,D) \mid \theta^X_a(t;0),\theta^H_0(t;0)\in (\beta,\omega)\}\\
A_2&=\{ t\in (-D,D) \mid \theta^X_a(t;0),\theta_0^H(t;0)\notin (\beta,\omega) \}.
\end{align*}
By the proof of Theorem \ref{first1},
\begin{align*}
\Ha^{1}(\R \backslash A_1 \cup A_2)={}& |t^X(a,\beta;0)-t^H(0,\beta;0)|+|t^X(a,\omega;0)-t^H(0,\omega;0)|\\
\leq {} &M_1 a
\end{align*}
where $M_1$ can be chosen independently of $x$ by $r$-regularity. Moreover,
\begin{align*}
|\theta_a^X(t;0)- \theta_0^H(t;0)| &\leq \max\{| \theta_a^{B_{in}}(t;0)- \theta_0^H(t;0)|,| \theta_a^{\R^d \backslash B_{out}}(t;0)- \theta_0^H(t;0)|\}\\
 &\leq a^{-d}\sup \rho \int_{aB^{d-1}(D)} (r - \sqrt{r^2-|z|^2})dz \\
 &\leq M_2 a
\end{align*}
where $M_2$ is again uniform by $r$-regularity. Hence
\begin{equation*}
a^{-1} \bigg|\int_{-D}^D (f\circ \theta_a^X(t;0)-f\circ \theta_0^H(t;0)) dt \bigg|\leq 2D \sup |f'|M_2 + \sup |f| M_1
\end{equation*}
and thus we can apply Lebesgue's theorem to the $m=1$ term as well if only we can determine the limit of
\begin{equation}\label{Xint}
a^{-1} \int_{-D}^D (f\circ \theta_a^X(t;0)-f\circ \theta_0^H(t;0)) dt .
\end{equation}

By Lemma \ref{tXtQ},
\begin{align*}
\MoveEqLeft a^{-1}\int_{-D}^D | f\circ \theta_a^X(t;0) -  f\circ \theta_a^Q(t;0) | dt \\ 
& \leq a^{-1}\int_{-D}^D (\sup |f'| |\theta_a^X(t;0) -  \theta_a^Q(t;0)|\mathds{1}_{B_1} + \sup |f|  \mathds{1}_{ \R \backslash (B_1\cup B_2)}) dt\\
&\leq 2DM_3a^{-1}\lambda(a)  + a^{-1}\sup |f| \Ha^1( \R \backslash (B_1\cup B_2))\\
&\leq 2DM_3a^{-1}\lambda(a)  + M_4 a^{-1}\lambda(a).
\end{align*}
for some $M_3,M_4>0$ where
\begin{align*}
B_1&=\{t\in (-D,D) \mid \theta^X_a(t;0),\theta^Q_a(t;0)\in (\beta,\omega)\}\\
B_2&=\{ t\in (-D,D) \mid \theta^X_a(t;0),\theta_a^Q(t;0)\notin (\beta,\omega) \}.
\end{align*}
Hence the limit of \eqref{Xint} equals the limit of
\begin{equation*}
 a^{-1} \int_{-D}^D (f\circ \theta_a^Q(t;0)  - f\circ \theta_0^H(t;0) ) dt.
\end{equation*}

As in the proof of Lemma \ref{tHtQ}, $\theta_a^Q(t;0)$ extends to a $C^1$ function $ \beta(a,t;0)$ for $(a,t)\in \R^2$.
Introduce the following sets:
\begin{align*}
C_1(a)&=\{t\in (-D,D) \mid \beta(a,t;0),\beta(0,t;0)\in (\beta,\omega)\}\\
C_\beta^1(a)&=\{t\in (-D,D) \mid  \beta(0,t;0)\leq \beta < \beta(a,t;0) \}\\
C_\beta^2(a)&=\{t\in (-D,D) \mid  \beta(a,t;0)\leq \beta < \beta(0,t;0) \}.
\end{align*}
First consider
\begin{equation*}
|f \circ \beta(a,t;0) - f \circ \beta(0,t;0)|\mathds{1}_{C_1(a)} \leq  a\sup{|f'|}\sup_{(a,t)\in B }\Big\{\frac{d}{da}\beta (a,t;0)  \Big\} \mathds{1}_{C_1(a)}
\end{equation*}
where $B$ is a compact neighborhood of $\{0\}\times \varphi([\omega,\beta],n)$.
It follows that 
\begin{align*}
\MoveEqLeft \lim_{a\to 0} a^{-1} \int_{C_1(a)}( f \circ \theta_a^Q(t;0) - f\circ \theta_0^H(t;0)) dt \\ 
& = \int_{\R} \lim_{a\to 0} \mathds{1}_{C_1(a)}a^{-1} (f\circ \beta(a,t;0)  -  f\circ \beta(0,t;0)) dt \\
&=\int_{\R} \mathds{1}_{C_1(0)} \frac{d}{da}f \circ \beta(a,t;0)_{\mid a=0} dt\\
&=-\frac{1}{2}\int_{\varphi((\beta,\omega),n)} f'(\theta^H_0(t;0) ) \int_{n^\perp} \II(z)\rho(z-tn) dz dt. 
\end{align*}


Next $C^1_\beta(a) \cup C^2_\beta(a)$ is the interval
\begin{equation*} [t^H(0,\beta;0)\mathds{1}_{C_\beta^1(a)}+t^Q(a,\beta;0)\mathds{1}_{C_\beta^2(a)},t^H(0,\beta;0)\mathds{1}_{C_\beta^2(a)}+t^Q(a,\beta;0)\mathds{1}_{C_\beta^1(a)}]
\end{equation*}
and 
\begin{equation*}
\mathds{1}_{C^1_\beta(a) \cup C^2_\beta(a)} (f \circ \beta(a,t;0) - f \circ \beta(0,t;0))=f \circ \beta(a,t;0)\mathds{1}_{C_\beta^1(a)} - f \circ \beta(0,t;0)\mathds{1}_{C_\beta^2(a)}.  
\end{equation*}

Let $\eps >0$ and $t \in C_\beta^1(a)$. Then
\begin{equation*}
|t-t^H(0,\beta;0)| \leq |t^H(0,\beta;0)-t^Q(a,\beta;0)| \leq M_4a
\end{equation*}
and hence
\begin{equation*}
|f \circ \beta(a,t;0) - f \circ \beta(0,t;0) - f_+(\beta)| =|  f \circ \beta(a,t;0)-f_+(\beta) |\leq \frac{\eps}{M_4}  
\end{equation*}
for $a$ sufficiently small, 
since $\beta$ is continuous. A similar argument for $t \in C_\beta^2(a)$ yields
\begin{align*}
\MoveEqLeft a^{-1}\bigg|\int_{C^1_\beta(a)\cup C^2_\beta(a)}( f \circ \beta(a,t;0) - f \circ \beta(0,t;0) - f_+(\beta) (\mathds{1}_{C_\beta^1(a)}(t)-\mathds{1}_{C_\beta^2(a)}(t)))dt\bigg|\\
&=a^{-1}\bigg|\int_{t^H(0,\beta;0)}^{t^Q(a,\beta;0)} (f \circ \beta(a,t;0) + f \circ \beta(0,t;0) - f_+(\beta) )dt\bigg| \\
&\leq a^{-1}\bigg| \int_{t^H(0,\beta;0)}^{t^Q(a,\beta;0)} \frac{\eps}{M_4} dt\bigg| \\
&\leq \eps
\end{align*}
for all $a$ sufficiently small.
It follows that
\begin{align*}
\MoveEqLeft \lim_{a\to 0} a^{-1}\int_{C_1(a)\cup C_2(a)} (f \circ \beta(a,t;0) - f \circ \beta(0,t;0)) dt = 
\lim_{a\to 0}a^{-1}\int_{t^H(0,\beta;0)}^{t^Q(a,\beta;0)} f_+(\beta) dt \\
&=\lim_{a\to 0}a^{-1}(t^Q(a,\beta;0)-t^H(0,\beta;0)) f_+(\beta) \\
&= \psi^Q(\beta;0)f_+(\beta).
\end{align*}
The $\omega$ terms are handled similarly.
\end{proof}

\section{Applications of the second order formula}\label{app2}
\subsection{Thresholding}\label{thri2}
In the case where a grey-scale image is thresholded at level $\beta$, Theorem \ref{second} reduces to:
\begin{corollary}
Let $(B_l,W_l)$ be a configuration. 
Let $X$ and $\rho$ be as in Theorem \ref{second} and let $\beta \in (0,1)$ be a regular value for $\theta_0^{H_{n}}(\cdot ;c)$ for all $c\in B_l\cup W_l$ and $n\in S^{d-1}$. Then 
\begin{align*}
\MoveEqLeft EN(\beta)_l^{a\La_c}(X) = a\int_{ \partial X } (-h(B_l\oplus \check{W_l},n))^+ d\Ha^{d-1} \\
&+ a^{2}\int_{ \partial X}\bigg( \Big(\tfrac{1}{2}((\varphi^{\rho}(\beta,n)-h({B}_l,n))^2 - (\varphi^{\rho}(\beta,n)+h(\check{W}_l,n))^2 )\tr \II \\
&\quad +\min_{b \in B_l^+(n)}\{\psi^{Q}(\beta;b)\}-\max_{w\in W_l^-(n)}\{\psi^{Q}(\beta; w)\}\Big) \mathds{1}_{ \big\{ h(B_l\oplus \check{W_l},n)<0 \big\}}\\
&\quad + \Big(\min_{b \in B_l^+(n)}\{\psi^{Q}(\beta;b)\}-\max_{w\in W_l^-(n)}\{\psi^{Q}(\beta; w)\}\Big)^+\mathds{1}_{ \big\{ h(B_l\oplus \check{W_l},n)=0 \big\}}\bigg)d\Ha^{d-1} \\
&+o(a^2).
\end{align*}
\end{corollary}
Here $S^\pm(n)$ is short for the support set $\{s\in S\mid h(S,\pm n )=\langle s, \pm n \rangle\}$. Comparing with the formula \cite[Theorem 4.3]{am2} for the black-and-white case, the first order term is the same, whereas the second order term now depends on $\rho$ and $\beta$. However, if $\rho $ is reflection invariant and $\beta = \frac{1}{2}$, then $\varphi^{\rho}(\beta,n)=0$ and the first line in the second order term is the same as in the black-and-white case. 

If $\supp \rho \subseteq B(D)$ and $c\in B_l\cup W_l$, then $\varphi^\rho(\beta,n)\in (-D,D)$ and hence
\begin{align*}
|\psi^{Q(x),\rho}(\beta;c)+\tfrac{1}{2} \II_x(c)|\leq \tfrac{1}{2}r^{-1}(D^2 + 2dn^2D).
\end{align*}
Thus, if $\rho$ is concentrated near $0$, so that $\theta^X(z)$ approximates the Dirac measure $\delta_z(X)$, the formula is close to the corresponding formula in the black-and-white case.

\subsection{First order bias of surface area estimators}
For surface area estimators, Theorem \ref{fest} yields:
\begin{corollary}
Let $X$, $f$, and $\rho$ be as in Theorem \ref{fest}. Then a first order expansion of
\begin{align*}
 E\hat{V}(f)_{d-1}^{a\La_c}(X) = a^{-1}\int_{\R^d} f\circ \theta_a^X d\Ha^{d-1} 
\end{align*}
is given by Theorem \ref{fest}.
\end{corollary}
In particular if $\rho $ is reflection invariant, then $1-\theta^{H_n}_0(t;0)=\theta^{H_n}_0(-t;0)$. If, moreover, $f$ satisfies $f(x)=f(1-x)$, 
\begin{align*}
\int_{(0,1)}f d\nu_n&=0\\
f'(\theta_0^{H_{n}}(t;0))&= -f'(\theta_0^{H_{n}}(-t;0))\\
f_+(\beta)\psi^{Q(x)}(\beta;0)&=f_-(1-\beta)\psi^{Q(x)}(1-\beta;0).
\end{align*}
Thus the second order term in Theorem \ref{fest} vanishes. This yields:

\begin{corollary} \label{even}
For $X$, $f$, and $\rho$ as in Theorem \ref{fest} with $\rho$ reflection invariant and $f(x)=f(1-x)$,
 \begin{equation*}
E\hat{V}(f)_{d-1}^{a\La_c}(X) = \int_{(0,1)} f d\mu_X + o(a). 
\end{equation*}
\end{corollary}
Recall that the condition $f(x)=f(1-x)$ was already justified by Corollary \ref{asworst} in order to minimize the asymptotic  bias.  

\begin{example}
Assume $\rho $ is rotation invariant. Under the assumptions of Theorem \ref{fest}, Corollary \ref{even} shows that for the asymptotically unbiased estimators \eqref{countbw}, choosing $\omega=1-\beta$ yields the best approximations in finite high resolution. These estimators take the form $\varphi^\rho(\beta)^{-1}N_{(\beta,1-\beta)}^{a\La_c}(X)$.  
\end{example}

\subsection{Estimation of the integrated mean curvature} \label{imc}
Similarly, for estimators for $V_{d-2}$, we obtain:
\begin{corollary}
Let $X$, $f$, and $\rho$ be as in Theorem \ref{fest}. Then 
\begin{align*}
\MoveEqLeft E\hat{V}(f)_{d-2}^{a\La_c}(X) = a^{-1}\int_{(0,1)} f d\mu_X\\
 &+ \int_{\partial X}\bigg(\tr \II\int_{(0,1)}f d\nu_n - \frac{1}{2} \int_\R f'(\theta_0^{H_{n}}(t;0))\int_{n^\perp} \II(z)\rho(z-tn)dz dt \\ 
&\quad +  f_+(\beta)\psi^{Q}(\beta;0)-f_-(\omega)\psi^{Q}(\omega;0)\bigg) d\Ha^{d-1}\\
&+o(1). 
\end{align*}
In particular, $\lim_{a\to 0} E\hat{V}(f)_{d-2}^{a\La_c}(X)$ exists if $\rho $ is reflection invariant and $f$ satisfies $f(x)=-f(1-x)$.
\end{corollary}
Suppose $\rho$ is rotation invariant and $f$ is as in Theorem \ref{fest} with $f(x)=-f(1-x)$. In particular,  $\omega =1-\beta$ for some $\beta \in (0,\frac{1}{2})$. Then we have 
\begin{align*}
\MoveEqLeft  \int_{\varphi(1-\beta)}^{\varphi(\beta)} f'(\theta_0^{H}(t;0))\int_{n^\perp} \II(z)\rho(z-tn)dz dt\\
&= \int_{-\varphi(\beta)}^{\varphi(\beta)} f'(\theta_0^{H}(t;0)) \int_0^\infty \int_{S^{d-2}} \II(u)r^2 \rho_t(r) r^{d-2} \Ha^{d-2}(du)dr dt\\
&= \kappa_{d-1} \tr \II  \int_{-\varphi(\beta)}^{\varphi(\beta)} f'(\theta_0^{H}(t;0)) \int_{n^\perp} |z|^2\rho(z-tn) dzdt
\end{align*}
where for a fixed $t\in \R$, $\rho_t$ is the function $\rho(z-tn)=\rho_t(|z|)$ for $z\in n^\perp$. Moreover, 
\begin{align*}
\MoveEqLeft f_+(\beta)\psi^{Q}(\beta;0)-f_-(1-\beta)\psi^{Q}(1-\beta;0)=2f_+(\beta)\psi^{Q}(\beta;0)\\
&=f_+(\beta)\varphi'(\beta)\int_{0}^\infty \int_{S^{d-2}}   \II(u)r^2  \rho_{\varphi(\beta)}( r )r^{d-2} \Ha^{d-2}(du)dr\\
&=\kappa_{d-1} \tr \II f_+(\beta)\varphi'(\beta)\int_{n^\perp}  |z|^2 \rho(z -\varphi(\beta)n) dz .
\end{align*}
Introducing the constants
\begin{align*}
c_1&=\int_{(0,1)}f d\nu_n=\int_{-\varphi(\beta)}^{\varphi(\beta)} t f\circ \theta^{H}_0(t;0) dt\\
c_2&=\kappa_{d-1} f_+(\beta)\varphi'(\beta)\int_{n^\perp}  |z|^2 \rho(z -\varphi(\beta)n) dz \\
c_3&=-\frac{\kappa_{d-1} }{2}\int_{-\varphi(\beta)}^{\varphi(\beta)} f'(\theta_0^{H}(t;0)) \int_{n^\perp}  |z|^2\rho(z-tn) dz dt,
\end{align*}
we obtain:
\begin{corollary}
For $X$, $f$, and $\rho$ as in Theorem \ref{fest} with $\rho $ rotation invariant and $f(x)=-f(1-x)$,
\begin{equation*}
\lim_{a\to 0} E\hat{V}(f)_{d-2}^{a\La_c}(X)  = (c_1+c_2+c_3) \int_{\partial X}\tr \II d\Ha^{d-1}  = 2\pi(c_1+c_2+c_3) V_{d-2}(X)
\end{equation*}
In particular, $\hat{V}(f)_{d-2}$ is asymptotically unbiased if and only if $c_1+c_2+c_3=(2\pi)^{-1}$.
\end{corollary}

\begin{example}
Let $f(x)=(x-\frac{1}{2})\mathds{1}_{(\beta, 1-\beta)}$. Then
\begin{align*}
c_1&=\int_{-\varphi(\beta)}^{\varphi(\beta)} t (\theta^H_0(t;0)-\tfrac{1}{2}) dt\\
&= \int_{-\varphi(\beta)}^{\varphi(\beta)} t\int_{-\infty}^{-t} \int_{n^\perp} \rho(z+sn)dzds dt\\
&= \int_{-\infty}^{\varphi(\beta)}\int_{-\varphi(\beta)}^{(-s)\wedge \varphi(\beta)}  tdt \int_{n^\perp} \rho(z+sn)dz ds\\
&= \int_{-\varphi(\beta)}^{\varphi(\beta)}\tfrac{1}{2}(s^2-\varphi(\beta)^2) \int_{n^\perp} \rho(z+sn)dz ds\\
&= \frac{1}{2}\int_{-\varphi(\beta)}^{\varphi(\beta)}\int_{n^\perp} s^2\rho(z-sn) dz ds+ \varphi(\beta)^2 (\beta-\tfrac{1}{2}).
\end{align*}
It follows that 
\begin{equation*}
c_1+c_2+c_3 = d_1(\varphi(\beta))d_1'(\varphi(\beta))^{-1}d_2'(\varphi(\beta))-d_2(\varphi(\beta))
\end{equation*}
where
\begin{align*}
d_1(t)&=\tfrac{1}{2}(\theta^H_0(t;0)-\theta^H_0(-t;0))\\
d_2(t)&=\frac{1}{2}\int_{-t}^{t}\int_{n^\perp}(\kappa_{d-1} |z|^2-s^2)\rho(z-sn) dzds.
\end{align*}
However, $d_2'(0)>0$ and $d_2'(t)<0$ for $t^2 \geq \frac{D^2}{1+\kappa_{d-1} }$ where $\supp \rho \subseteq B(D)$. 
By continuity and the fact that $d_2(0)=0$, $d_2$ must have a local maximum at some $t_0\in \Big(0,\frac{D}{\sqrt{1+\kappa_{d-1} }}\Big)$ with $d_2(t_0)>0$. Hence $c_1 + c_2 + c_3 \neq 0$ for $\beta $ in some neighborhood of $ \beta_0=\theta^H_0(t_0;0)$. 

It follows that the function 
\begin{equation*}
f(x)=(2\pi(c_1+c_2+c_3))^{-1}\big(x-\tfrac{1}{2}\big)\mathds{1}_{(\beta_0, 1-\beta_0)}(x)
\end{equation*}
yields an asymptotically unbiased estimator for $V_{d-2}$. If $\rho $ is known, the constants $c_1$, $c_2$, $c_3$, and $\beta_0$ can be determined directly by the above, otherwise these constants could be determined experimentally.
\end{example}

\begin{example}
A similar argument shows that also the estimator $N_{(\beta,\frac{1}{2})}- N_{(\frac{1}{2},1-\beta)}$ is asymptotically unbiased up to some constant factor which is non-zero for a suitable $\beta \in (0,\frac{1}{2})$. This estimator has the same advantage as \eqref{countbw} that it can be applied even if the grey-values are only known discretely.
\end{example}

\section{Discussion}\label{discuss}
To judge from the results of this paper, it seems that the blurring of digital images should be considered a help rather than an obstacle to the estimation of intrinsic volumes. The biasedness of local algorithms in the black-and-white case can be viewed as a consequence of the rotational asymmetry of the $n\times \dotsm \times n$ pixel configurations when $n>1$. For $n=1$  there is only one estimator, namely the volume estimator, which is well known to be unbiased. In the grey-scale setting, choosing $n=1$, thus avoiding the asymmetry, leads to a wide range of estimators, allowing instead an exploitation of the symmetry of a rotation invariant PSF to obtain information about the lower intrinsic volumes.

One should keep in mind, however, that the results of this paper are only asymptotic and say nothing about how the suggested algorithms work in finite resolution. Especially because of the assumptions on the asymptotic behaviour of the PSF. Moreover, it is not possible to say much from the asymptotic results about which algorithms work best in practice.  For instance, it is not clear how to choose $\beta $ best possible for the estimator $N_{(\beta, 1-\beta)}$.  Thus local grey-scale algorithms should be carefully studied and tested in finite resolution before being taken into use.

In some practical applications it may be possible to adjust the PSF, for instance if the PSF has the form $\rho_B$. The results of this paper could be used to design measurements such that the suggested algorithms apply, for instance by choosing a PSF of the form $\rho_B$ with $B$ rotation invariant rather than the classical $\rho_{C_0}$. 

From the mathematical viewpoint, the proven existence of asymptotically unbiased estimators for intrinsic volumes $V_q$ with $q=d,d-1,d-2$ naturally raises the question whether it stops here or generalizes to the remaining $V_q$ with $q<d-2$. A proof would probably require some stronger smoothness assumptions on both $X$, $\rho$, and $f$ and maybe a whole different approach. 

\section*{Acknowledgements}
The author was supported by the Centre for Stochastic Geometry and Advanced Bioimaging, fund\-ed by the Villum Foundation.
The author is wishes to thank Markus Kiderlen for helping with the set-up of this research project and for useful input along the way.

\end{document}